\documentclass[11pt,reqno]{amsart}
\usepackage[utf8x]{inputenc}
\usepackage[T1]{fontenc}
\usepackage{graphicx}
\usepackage[font=footnotesize]{caption}
\usepackage{algorithm}
\usepackage[noend]{algpseudocode}
\usepackage{todonotes}
\usepackage{txfonts}
\usepackage{mathtools}

\date{}

\title[Likelihood-Based Root State Reconstruction on a Tree]{Likelihood-Based Root State Reconstruction on a Tree: Sensitivity to Parameters and Applications}
\author{
}

\usepackage[english]{babel}
\usepackage[margin = 1in]{geometry}
\usepackage{graphicx}

\usepackage[colorlinks= True, citecolor=blue]{hyperref}
\usepackage{amsthm}
\usepackage{amsfonts}
\usepackage{amssymb}
\usepackage{amsmath}
\usepackage{mathrsfs}
\usepackage{mathtools}
\usepackage{xcolor}
\usepackage{qtree}
\usepackage{subcaption}
\usepackage[]{mdframed}

\newtheorem{innercustomassumption}{}
\newenvironment{customassumption}[1]
{\renewcommand\theinnercustomassumption{#1}\innercustomassumption}
{\endinnercustomassumption}

\DeclareMathOperator*{\argmax}{arg\,max}

\newtheorem{theorem}{Theorem}[section]

\newtheorem{claim}[theorem]{Claim}
\newtheorem{lemma}[theorem]{Lemma}
\newtheorem{corollary}[theorem]{Corollary}

\theoremstyle{definition}
\newtheorem{def1}[theorem]{Definition}
\newtheorem{rmk1}[theorem]{Remark}

\newcommand{\E}{\mathbb{E}}

\newcommand{\PR}{\mathbb{P}}
\newcommand{\R}{\mathbb{R}}

\newcommand{\Z}{\mathbb{Z}}

\newcommand{\eps}{\varepsilon}

\renewcommand{\PR}{\mathbb{P}}

\newcommand{\param}{{\boldsymbol{\theta}}}
\newcommand{\Param}{{\boldsymbol{\Pi}}}

\newcommand{\htheta}{\hat{\theta}}
\newcommand{\hparam}{{\hat{\param}}}
\newcommand{\hParam}{{\widehat{\Param}}}

\usepackage{verbatim}

\begin{document}
	
	\author{David Clancy, Jr.}
	\address{David Clancy, Jr., Department of Mathematics, University of Wisconsin - Madison, WI, 53717, USA}
	\email{\texttt{dclancy@math.wisc.edu}}

	\author{Hanbaek Lyu}
	\address{Hanbaek Lyu, Department of Mathematics, University of Wisconsin - Madison, WI, 53717, USA}
	\email{\texttt{hlyu@math.wisc.edu}}

	\author{Sebastien Roch}
	\address{Sebastien Roch, Department of Mathematics, University of Wisconsin - Madison, WI, 53717, USA}
	\email{\texttt{roch@wisc.edu}}
	
	\author{Allan Sly}
	\address{Allan Sly, Department of Mathematics, Princeton University, 
		Princeton, NJ, 08540, USA}
	\email{\texttt{asly@princeton.edu}}

\begin{abstract}
We consider a broadcasting problem on a tree where a binary digit (e.g., a spin or a nucleotide's purine/pyrimidine type) is propagated from the root to the leaves through symmetric noisy channels on the edges that randomly flip the state with edge-dependent probabilities.
The  goal of the reconstruction problem is to infer the root state given the observations at the leaves only.
Specifically, we study the sensitivity of maximum likelihood estimation (MLE) to uncertainty in the edge parameters under this model, which is also known as the Cavender-Farris-Neyman (CFN) model. Our main result shows that when the true flip probabilities are sufficiently small, the posterior root mean (or magnetization of the root) under estimated parameters (within a constant factor) agrees with the root spin with high probability and deviates significantly from it with negligible probability. This provides theoretical justification for the practical use of MLE in ancestral sequence reconstruction in phylogenetics, where branch lengths (i.e., the edge parameters) must be estimated. As a separate application, we derive an approximation for the gradient of the population log-likelihood of the leaf states under the CFN model, with implications for branch length estimation via coordinate maximization.    
\end{abstract}
\maketitle

	\section{Introduction}
    
   We consider the following broadcasting problem on a tree $T=(V,E)$. Information is transmitted from the root $\rho$ to the leaves $L$ through the edges, which function as noisy channels. More formally, information is propagated according to a stochastic process where the root $\rho$ is initially assigned a spin $\sigma_\rho\in \{\pm 1\}$ uniformly at random and the spins are transmitted down toward the leaves, where each edge $e=(u,v)$ of $T$ randomly flips the sign of the spin with probability $p_e^*$, or put differently according to the following transition matrix:
    \begin{equation*}
        \begin{bmatrix}
			1-p^*_e & p^*_e\\
			p^*_e &1-p^*_e
		\end{bmatrix}
        =
        \frac{1}{2}
        \begin{bmatrix}
			1+\theta^*_e & 1-\theta^*_e\\
			1-\theta^*_e & 1+\theta^*_e
		\end{bmatrix}.
	\end{equation*}
    where it can be checked that $\theta^{*}_{e}$, a useful alternative parameterization of the model, is equal to the second eigenvalue of the transition matrix.
    The goal is to infer the root spin $\sigma_\rho$ given the spins $\{\sigma_x, x\in L\}$, at the leaves. 
     We assume that $p_e^* \in (0,1/2)$, or equivalently that $\theta^*_e \in (0,1)$, meaning that it is more likely that spins do not flip. 
     We will refer to the collection $\param^{*} = (\theta^*_e: e\in E)$ as the ``true parameter'' of the model.

This so-called ``reconstruction problem'' 
has been extensively studied. 
Two competing forces impact the
reconstructibility of the root state, leading to a phase transition: 
the noise on the edge channels produce a loss of information along the paths to the leaves; 
on the other hand, the branching duplicates the signal and creates a large number of correlated noisy processes to the leaves. In the special case where $T$ is a complete $b$-ary tree and $\theta_e^* = \theta$ for all $e$, a classical result states that, as the number of levels goes to infinity, there exists a root state estimator with success probability bounded below from $1/2$ if and only if $b \theta^2 > 1$~\cite{BlRuZa:95,Ioffe:96a,kesten1966additional}. See, e.g.,~\cite[Section 6.3.2]{Roch_2024} for an account of the proof. This is known as the Kesten-Stigum (KS) bound and has been generalized in a number of ways (e.g., \cite{bhatnagar2011reconstruction,BCMR.06,EvKePeSc:00,FanRoch:18,JansonMossel:04,makur_broadcasting_2020,Mossel:01,Mossel:04a,MosselPeres:03,Sly:11}). 
See also~\cite{mossel:22} for a recent survey and some applications.  

The present work is motivated in part by applications in biology, where the model above is called the Cavender-Farris-Neyman (CFN) model \cite{cavender1978taxonomy, farris1973probability, neyman1971molecular} and is used to study molecular evolution on an evolutionary tree. Here the transformed parameters $l^*_e = -\ln \theta_e^*$ are known as branch lengths.
The name ``branch length'' comes from the fact that one can view the flip probability $p^*_{e}$ as resulting from a continuous-time constant rate of mutation along the edge $e$ for an amount of time proportional to $l^*_{e}$. In that context the KS bound pertains to the reconstruction of ancestral sequences and has important implications for tree inference as well (e.g.,~\cite{DMR.11,Roch:10,Mossel:04a,RochSly:17}). 
Several different methods for ancestral reconstruction are used in the phylogenetics literature. Perhaps most common is the maximum likelihood estimator (MLE), which in this case boils down to estimating the root spin by the state with highest posterior probability given the observations at the leaves. Put differently, one infers a state according to the sign of the posterior root mean, also called magnetization of the root (i.e., the conditional bias under $\PR_{\hparam}$ given the leaf spins, see Def. \ref{def:magnetization}).  

That root state estimator has a flaw -- it depends on the branch lengths, which must be separately estimated. The central question we ask in this paper is the following:
\emph{How sensitive is the MLE to deviations in the branch lengths?} In a nutshell we show that, if the true parameter $\param^{*}$ is sufficiently deep inside the KS reconstructibility regime and the estimated parameter $\hparam$ is within a constant factor, the magnetization of the root agrees with the root spin with high probability and deviates significantly from it with negligible probability.

    Our results are not the first ones to show that a root state estimator is accurate without the precise knowledge of the true parameter. Indeed, for the CFN model on a binary tree, an estimator known as ``recursive majority'' does not use the branch lengths at all~\cite{Mossel:04a}; it also works all the way to the KS threshold $1/\sqrt{2}\approx 0.707$ (i.e., whenever $\theta_e^* \leq 1/\sqrt{2}$ for all $e$). But, unlike the MLE, it is \emph{not} used in practice. Our current bound for the MLE holds whenever
    \begin{align*}
        \hparam,\param^*\in [1-\delta, 1- \delta/2]^{E(T)}\qquad\textup{and}\qquad \delta\le \frac{1}{1190}.
    \end{align*} 
    See Remark \ref{rmk:constants}. This is quite far off from the KS bound. The extent to which this
    can be improved is an interesting open problem. 

Our results also have further implications in phylogenetics. 
The related branch length estimation problem 
is to recover the unknown true parameter $\param^{*}$ from repeated, independent observations of the states at the leaves of the tree under the true model $\PR_{\param^{*}}$. It is of practical interest
	whether standard optimization algorithms for the maximum likelihood approach, e.g., coordinate ascent~\cite{guindon2003simple}, can 
    accurately estimate the true parameter $\param^{*}$.
    Our new tail bounds on the magnetization provide new understanding of the likelihood landscape of the CFN model. This is because the gradient of the log-likelihood of the leaf states under the hypothetical model $\PR_{\hparam}$ admits a simple expression in terms of the magnetizations (see Lem. \ref{lem:derivative}). 
    As an application of our main result, we derive an approximation for the gradient of the population log-likelihood of the CFN model 
    (Thm. \ref{thm:gradient}). 
    Deeper results on the Hessian of the population log-likelihood in the neighborhood of the true parameter and an analysis of coordinate maximization in a finite-sample setting will be detailed in forthcoming papers.


    \begin{figure}[!ht]
    	\centering
    	\includegraphics[width=.8\textwidth]{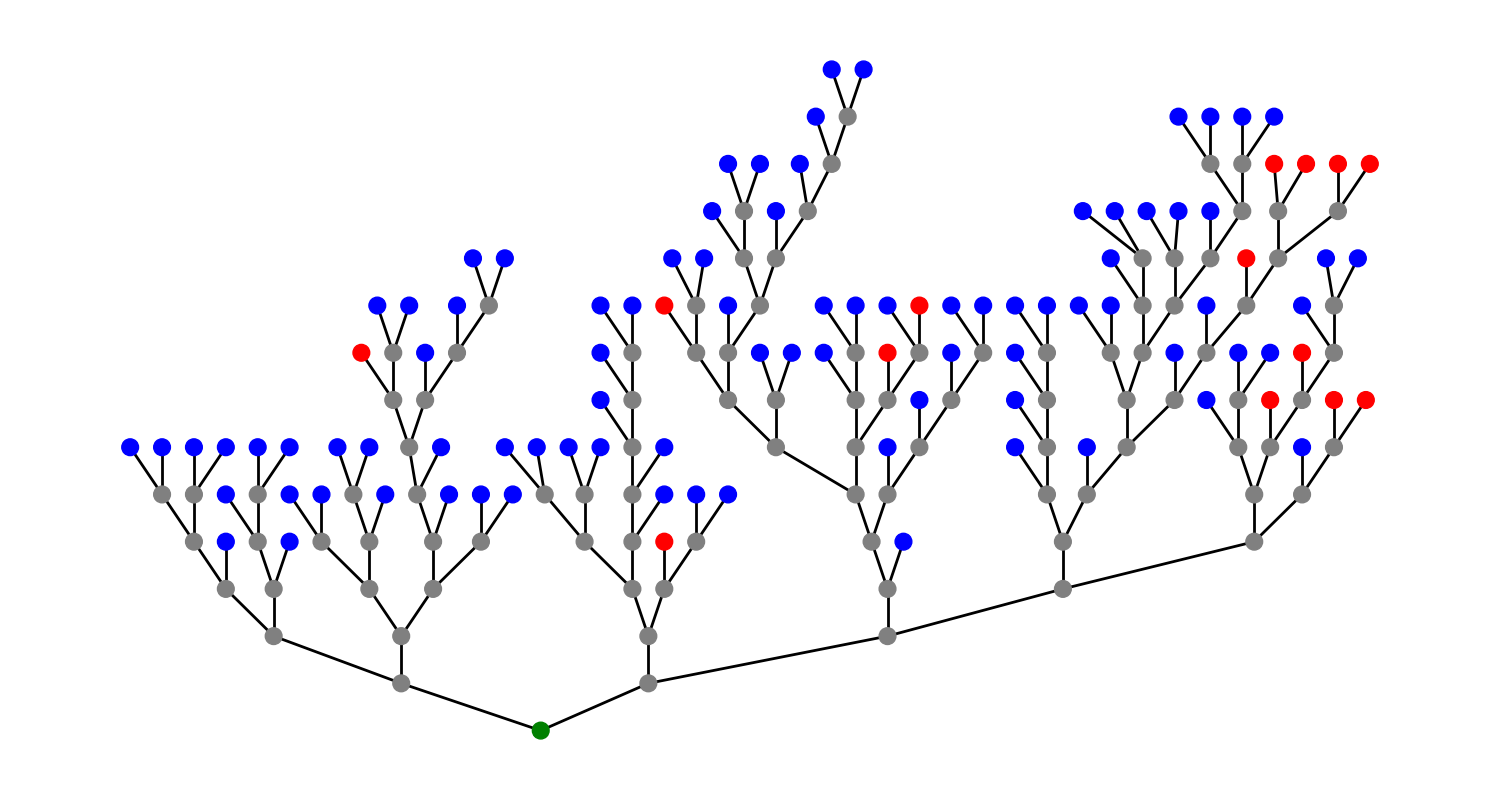}
    	\caption{Sample spin configuration on the leaves of a rooted binary tree $T$ with $n = 100$ leaves. The root is in green, blue leaves have a state $+1$ and red leaves have a state $-1$. The parameters $\theta^*_e, \htheta_e \overset{i.i.d.}{\sim}\textup{Unif}([0.9,0.95])$ and the root magnetization $Z_u \approx 0.997$.}
    	\label{fig:SignalReconstruct}
    \end{figure}

	\subsection{Overview and organization}

    This paper is organized as follows. In Section \ref{sec:main_results}, we first state the assumptions (\ref{assumption1}) and our main results (Thm.~\ref{thm:Robust}) on reconstruction. We then state an application to the gradient of the population log-likelihood (Thm.~\ref{thm:gradient}) and a corollary (Cor.~\ref{cor:coord_ini}) of on coordinate maximization for estimating the true parameter. In the next section, we give a proof of Theorem \ref{thm:Robust}. 
    In Section \ref{sec:keyLemmas}, we prove the application results. 

	\section{Statement of main results}
	\label{sec:main_results}
	
	In this section, we formally state the assumptions and main results.  Throughout the paper, constants are numbered by the equation in which they first appear.

	\subsection{Assumptions} 

    We assume that the tree $T = (V,E)$ is a finite, unrooted binary tree, that is, all its internal (i.e., non-leaf vertices) vertices have degree $3$.
    We will denote the number of leaves in $T$ by $n$. It can be checked that there is then $2n-3$ vertices in $T$ overall. 
    

	We introduce the parameter $\delta>0$ to control the scale of the mutation (or flip) probabilities along the edges. Our analysis operates under the assumption that we are well within the reconstruction regime, that is, that mutation probabilities are sufficiently small that ancestral states can be reconstructed with better-than-random accuracy.
		 We introduce the following restricted parameter spaces that depend on $\delta$. 

    \begin{def1}[Restricted parameter spaces]\label{def:parameter_spaces}
        Let $0 < c_{\ref{eqn:pHatBounds}} < c_{\ref{eqn:pBounds}} < C_{\ref{eqn:pBounds}} < C_{\ref{eqn:pHatBounds}} < 1$ be fixed constants. For $\delta>0$, define two subsets $\Param_{0}(\delta)\subseteq \hParam_{0}(\delta)\subset[-1,1]^{E}$  
        		by 
        		\begin{align}
        			\Param_{0}(\delta) &:=\left\{ (\theta^*_{e}=1-2p_{e}\,;\, e\in E) \,\bigg|\, c_{\ref{eqn:pBounds}} \delta \le p_e \le C_{\ref{eqn:pBounds}}\delta 
        			\textup{  $\forall$}e\in E \right\}=[1-2C_{\ref{eqn:pBounds}}\delta, 1-2c_{\ref{eqn:pBounds}}\delta]^{E}, \label{eqn:pBounds} \\
        			\hParam_{0}(\delta) &:=\left\{ (\htheta_{e}=1-2\hat{p}_{e}\,;\, e\in E) \,\bigg|\, c_{\ref{eqn:pHatBounds}}\delta \le \hat{p}_e \le C_{\ref{eqn:pHatBounds}} \delta 
        			\textup{  $\forall$}e\in E\right\} = [1-2C_{\ref{eqn:pHatBounds}}\delta, 1-2c_{\ref{eqn:pHatBounds}}\delta]^{E}. \label{eqn:pHatBounds}
        		\end{align}
    \end{def1}


    We assume the (unknown) true parameter $\param^{*}$ belongs to the $\delta$-box $\Param_{0}(\delta)$ and that the estimated parameter $\hparam$ at hand that we use to infer the root spin from the leave spins belongs to the larger $\delta$-box $\hParam_{0}(\delta)$. This assumption is stated in \ref{assumption1} below. 
    
	\begin{customassumption}{A1}[Parameter regime]\label{assumption1} 
		Assume that $\param^{*}\in \Param_{0}(\delta)$ and $\hparam\in \hParam_{0}(\delta)$. Moreover, the constants $c_{\ref{eqn:pHatBounds}}, C_{\ref{eqn:pHatBounds}}$ satisfy $C_{\ref{eqn:pHatBounds}}\ge 2c_{\ref{eqn:pHatBounds}}$.
	\end{customassumption}
	

\subsection{Posterior-based ancestral state reconstruction: insensitivity to parameters}


Fix two distinct nodes $u,v$ in $T$. We call a node $w$ a \textit{descendant} of $u$ with respect to node $v$ if the shortest path between $w$ and $v$ contains $u$. The \textit{descendant subtree at $u$} with respect to $v$ is the subtree $T_{u}$ rooted at $u$ consisting of all descendants of $u$ with respect to $v$. A subtree of $T$ rooted at $u$ is a \textit{descendant subtree of $u$} if it is a descendant subtree of $u$ with respect to some node $v$.

	Our central quantity of interest is the posterior distribution of the state of an internal node given the leaf states in a descendant subtree. Because the states take only two values here, it is enough to consider the \emph{posterior mean}, which is often referred to as ``magnetization.''  
 
 Roughly speaking, the magnetization  
 at a node $u$ with respect to a descendant subtree $T_{u}$ rooted at $u$ is the ``bias'' on its spin after observing all spins at the leaves of the descendant subtree $T_{u}$. For instance, if all spins on the leaves of $T_{u}$ are $+1$, then $u$ will be quite likely to have a $+1$ spin as well. The formal definition of the magnetization is given below and see Figure \ref{fig:SignalReconstruct} for a
 concrete example.
	\begin{def1}[Magnetization]\label{def:magnetization}
		Let $T_{u}$ be a descendant 
		subtree of $T$ rooted at a node $u$. Let $L_{u}$ denote the set of all leaves in $T_{u}$. For a generic parameter $\hparam \in [-1,1]^{E(T_u)}$ and fixed spin configuration $\sigma_{L_u}\in \{\pm 1\}^{L_{u}}$
  on the leaves of $T_u$, define the magnetization at the root $u$ of $T_{u}$ under $\hparam$ as
		\begin{align}\label{eqn:def_magnetization}
   Z^{\hparam,T_u}_u (\sigma_{L_{u}}):=
			\PR_\hparam(\hat{\sigma}_u = +1\, |\,  \hat{\sigma}_{L_{u}} = \sigma_{L_{u}}) - \PR_\hparam(\hat{\sigma}_u=-1\, |\,  \hat{\sigma}_{L_{u}} = \sigma_{L_{u}}),
		\end{align}
		where $\hat{\sigma}$ is a random spin configuration on $T$  sampled from $\PR_{\hparam}$. 
		Furthermore, if $\sigma$ is a random spin configuration sampled from $\PR_{\param^{*}}$, we consider the random variable 
		\begin{align*}
  Z_{u}^{\hparam,T_{u}}:=Z_{u}^{\hparam, T_{u}}(\sigma_{L_{u}}). 
		\end{align*}
  We write this random variable simply as $Z_u$ when $\hparam,T_{u}$ are clear from the context. 
	\end{def1}

     Recall that our overarching goal is to infer the unobserved spin $\sigma_u$ under $\PR_{\param^{*}}$ given the observed leaf spins $\sigma_{L_{u}}$ on a descendant subtree $T_u$ of a known tree $T$. The estimator we seek to analyze is simply the magnetization $Z_u^{\hparam,T_{u}}$ 
    \emph{using the estimated parameter $\hparam$}. We think of the latter as being deterministic and obtained externally---in particular, they do not depend on the observations. Fix $u$, $T_u$, and $\hparam$. Roughly, we wish to establish conditions under which 
    the following random approximation holds
    \begin{align}
        \sigma_u \approx Z_u. 
    \end{align}
    Since $\sigma_u \in \{\pm 1\}$ and $Z_u \in [-1,1]$, the accuracy of the inference above can be quantified by the following quantity 
    \begin{align}
        \sigma_u Z_u \in [-1,1]. 
    \end{align}
    which we refer to as ``unsigned magnetization.'' If this is close to 1, then the estimation is accurate; if it is close to $-1$, then it gives almost the opposite answer. 

 Magnetizations are straightforward to compute efficiently even for a large tree by using a recursive formula 
 (see, e.g., \cite{BCMR.06}). The idea is to compute the magnetizations from the leaves of $T_u$ and work towards $u$. 

    \begin{enumerate}
        \item (\textit{Base step}) For each leaf $\ell$ in $T_u$, set $Z_\ell=\sigma_\ell$ (since we get to observe the spin at $\ell$).
        \vspace{0.1cm}

        \item (\textit{Induction step}) Suppose $T_{v}$ is the descendant subtree of node $v$ with respect $u$ and let $w_1, w_2$ be its two children in $T_{v}$. There are corresponding descendant subtrees rooted at these nodes with respect to $u$, which defines the magnetization at these nodes, $Z_{w_1}$ and $Z_{w_2}$. 
Then \cite[Lemma 4 and 5]{BCMR.06} imply that 
\begin{equation}\label{eqn:recursionBorg}
	Z_v = \frac{\hat{\theta}_{w_1}Z_{w_1} + \hat{\theta}_{w_2} Z_{w_2}}{1+ \hat{\theta}_{w_1}\hat{\theta}_{w_2} Z_{w_1} Z_{w_2}}, 
\end{equation} 
where $\hat{\theta}_{w_1} := \hat{\theta}_{\{v,w_1\}}$ for the edge $\{v,w_1\}$ and similarly for $w_2$.
    \end{enumerate}

Now we state our main result in this work, Theorem \ref{thm:Robust}. Roughly speaking, it states that for every node $u$ in $T$, under $\PR_{\param^{*}}$, its spin $\sigma_u$ is close to its magnetization $Z_{u}$ computed with respect to any descendant subtree $T_{u}$ under estimated parameters satisfying certain conditions. More precisely, $Z_{u}$ is ``close'' to the spin $\sigma_{u}$ at $u$ with probability $1-O(\delta)$ irrespective of the parameters used as long as they are in the ``right range,'' and it will be off ``moderately'' with probability $\Theta(\delta)$ and ``severely'' with probability $O(\delta^{2})$. In the statement, recall that $\PR_{\param^{*}}$ indicates that the random spin configuration $\sigma_{L_{u}}$ on the leaves in $L_{u}$ is generated by the model with parameter $\param^{*}$.



\begin{theorem}[Insensitivity of magnetization to parameters]\label{thm:Robust} There exist constants $\delta_{\ref{eqn:antiReconstruction}}, 
c_{\ref{eqn:Reconstruct}},C_{\ref{eqn:Reconstruct}}, c_{\ref{eqn:antiReconstruction}}, C_{\ref{eqn:antiReconstruction}}>0$ depending only on the constants in \ref{assumption1} such that the following holds for any unrooted binary tree $T$ and $\delta\le \delta_{\ref{eqn:antiReconstruction}}$. 
	Fix a descendant subtree $T_{u}$ of a node $u$, and let $L_{u}$ denote the set of all leaves in $T_{u}$ and suppose that $\hparam\in \hParam_0(\delta)$ and $\param\in \Param_0(\delta)$. 
    \vspace{0.1cm}
	\begin{description}
		\item[(i)] (\textit{Upper tail}) 
        \vspace{-0.6cm} 
		\begin{equation}\label{eqn:Reconstruct}
			\PR_{\param^{*}}\left(\sigma_u Z_u^{\hparam, T_{u}}(\sigma_{L_{u}}) 
   \ge 1-C_{\ref{eqn:Reconstruct}} \delta^2  \right) \ge 1- c_{\ref{eqn:Reconstruct}}\delta. 
		\end{equation} 
        \vspace{0.1cm}

		\item[(ii)] (\textit{Lower tail}) 
		\vspace{-0.6cm}
    \begin{equation}\label{eqn:antiReconstruction}
			\PR_{\param^{*}} \left(\sigma_u Z_u^{\hparam, T_{u}}(\sigma_{L_{u}}) 
   \le - c_{\ref{eqn:antiReconstruction}}  \right) \le C_{\ref{eqn:antiReconstruction}}\delta^2.
		\end{equation}
	\end{description}
\end{theorem} 

	Following Theorem \ref{thm:Robust}, we introduce a trichotomy for the unsigned magnetizations. We say we have a ``good reconstruction'' at node $u$ if $\sigma_{u}Z_{u} \ge 1-C_{\ref{eqn:Reconstruct}} \delta^2$, ``severe failure'' if $\sigma_{u}Z_{u}\le  -c_{\ref{eqn:antiReconstruction}}$, and ``moderate failure'' otherwise. Then Theorem \ref{thm:Robust} states the following probability bounds for each of the three tiers of magnetization. See Figure \ref{fig:ZuHist} for a numerical experiment illustrating Theorem \ref{thm:Robust}.  
	\begin{align}\label{eq:recons_tiers}
		\begin{cases}
			\textup{Good reconstruction at $u$} \quad \Longleftrightarrow \quad 	\sigma_{u}Z_{u} \ge 1-C_{\ref{eqn:Reconstruct}} \delta^2 &\textup{(with prob. $\ge 1- c_{\ref{eqn:Reconstruct}}\delta$)} \\
			\textup{Moderate failure at $u$}  \quad \Longleftrightarrow \quad 	\sigma_{u}Z_{u}\in ( -c_{\ref{eqn:antiReconstruction}}, 1-C_{\ref{eqn:Reconstruct}}\delta^2) &\textup{(with prob. $\le c_{\ref{eqn:Reconstruct}}\delta$)} \\
			\textup{Severe failure at $u$}  \quad \Longleftrightarrow \quad 	\sigma_{u}Z_{u}\le  -c_{\ref{eqn:antiReconstruction}} 
			&\textup{(with prob. $\le C_{\ref{eqn:antiReconstruction}}\delta^{2}$)}
		\end{cases}
	\end{align}


\begin{figure}[!ht]
	\centering
	\includegraphics[width=
	\textwidth]{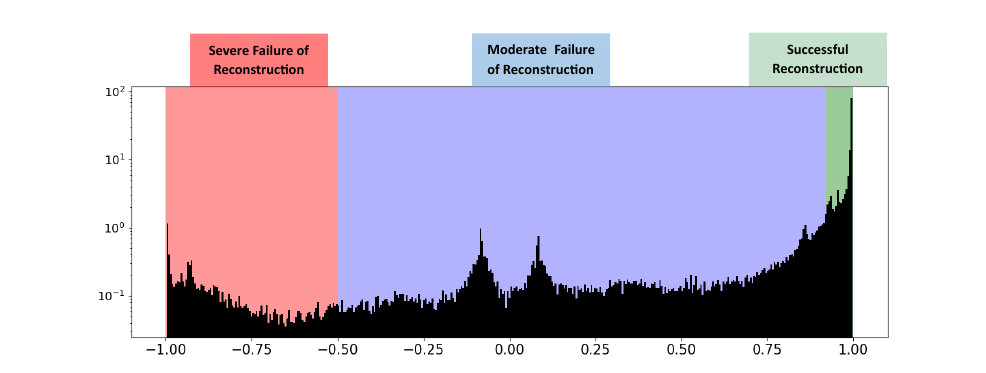}
	\caption{Histogram the 100,000 samples of unsigned magnetization $\sigma_uZ_u$ at the root of a descendant subtree $T_u$ with $n = 1000$ leaves. The horizontal axis is the unsigned magnetization, and the vertical axis is the normalized frequency plotted on a logarithmic scale. The parameters $\theta^*_e, \htheta_e \overset{i.i.d.}{\sim}\textup{Unif}([0.9,0.95])$. The two spikes around $\sigma_u Z_u \approx \pm 0.1$ correspond roughly to a flip on the edge from $u$ to one of its two children. The spike close to $\sigma_uZ_u\approx -1$ corresponds roughly to a flip on \textit{both} edges from $u$ to its children. 
		The red region represents \textit{severe failures} of reconstruction, the {blue}
		region represents \textit{moderate failures} of reconstruction, and the green region correspond to \textit{(successful) reconstruction}. 
	}
	\label{fig:ZuHist}
\end{figure}

\begin{rmk1}\label{rmk:constants}
The constants 
$c_{\ref{eqn:Reconstruct}},C_{\ref{eqn:Reconstruct}}, c_{\ref{eqn:antiReconstruction}}, C_{\ref{eqn:antiReconstruction}}$ in Theorem \ref{thm:Robust} are given explicitly by  $c_{\ref{eqn:Reconstruct}} = 7C_{\ref{eqn:pBounds}}$,   $C_{\ref{eqn:Reconstruct}}= \frac{4}{5}\big(\frac{9 C_{\ref{eqn:pHatBounds}}^2}{c_{\ref{eqn:pHatBounds}}} + 2 C_{\ref{eqn:pHatBounds}}\big)^2$, $c_{\ref{eqn:antiReconstruction}}  = 1 - \frac{2 c_{\ref{eqn:pHatBounds}}}{3C_{\ref{eqn:pHatBounds}}}$, and $C_{\ref{eqn:antiReconstruction}} =  78 C_{\ref{eqn:pBounds}}^2$. By chasing down the proof of Theorem \ref{thm:Robust}, one can check that 
\begin{align*}
    \delta_{\ref{eqn:antiReconstruction}} = \frac{1}{2380 C_{\ref{eqn:pBounds}}} \wedge \frac{C_{\ref{eqn:pHatBounds}}}{2C_{\ref{eqn:Reconstruct}}}\wedge \frac{5}{72C_{\ref{eqn:pHatBounds}}} \wedge c_{\ref{eqn:pHatBounds}}.
\end{align*} 
In particular, for the choices of $C_{\ref{eqn:pBounds}} =C_{\ref{eqn:pHatBounds}} =1/2$ and $c_{\ref{eqn:pBounds}} =c_{\ref{eqn:pHatBounds}} =1/4$, the conclusions of Theorem \ref{thm:Robust} hold with
    \begin{align*}
        C_{\ref{eqn:Reconstruct}} = 80, \qquad &c_{\ref{eqn:Reconstruct}} = \frac72, \qquad  c_{\ref{eqn:antiReconstruction}} = \frac23, \qquad C_{\ref{eqn:antiReconstruction}} = \frac{39}{2}, \qquad \delta_{\ref{eqn:antiReconstruction}} = \frac{1}{1190}.
    \end{align*}
\end{rmk1}

\subsection{Application: the gradient of the population log-likelihood on a known phylogeny}
     
    To demonstrate the utility of Theorem \ref{thm:Robust}, we consider the branch length estimation problem using the maximum likelihood framework.
    Fix an unrooted binary tree $T=(V,E)$---the known phylogeny.

    The log-likelihood of $m$ samples of the leaf observations $\sigma^{(1)},\dots,\sigma^{(m)}$ under the model with parameter $\hparam$ over $T$ is given by 
	\begin{align}\label{eq:def_log_likelihood}
	\ell(\hparam):=\ell(\hparam; \sigma^{(1)},\dotsm, \sigma^{(m)}) := \frac{1}{m} \sum_{j=1}^m  \log \PR_{\hparam}(\hat\sigma_v = \sigma^{(j)}_v , \,\,\forall v\in L),
	\end{align} 
 where $\hat{\sigma}$ is a random spin configuration on $T$  sampled from $\PR_{\hparam}$. 
	We then seek to find the maximum likelihood estimator (MLE) of the true parameter $\param^{*}$ as 
	\begin{align}\label{eq:MLE_def}
		\hparam_{\textup{MLE}} \in \argmax_{\hparam\in [-1,1]^{E}}  \,\,  \ell(\hparam; \sigma^{(1)},\dotsm, \sigma^{(m)}),
	\end{align}
	where now the leaf observations $\sigma^{(1)},\dots,\sigma^{(m)}$ are i.i.d.~samples from $\PR_{\param^{*}}$ (on the known phylogeny $T$).

  While the log-likelihood function $\ell$ in \eqref{eq:def_log_likelihood} is in general non-concave, it has the nice property of 
  having at most one stationary point
  when restricted to a single edge  $e\in E$ (unless it is constant) \cite{fukami_maximum_1989,dinh_matsen_2017}. Thus, it is natural to cycle through the branch lengths and optimize one of them at a time \cite{guindon2003simple}. 
  This yields the following ``cyclic coordinate maximization'' algorithm for computing the MLE branch lengths $\param^{*}$: given the current estimate $\hparam_{k}=(\htheta_{k;e};e\in E)$ after $k$ iterations, compute the new estimate $\hparam_{k+1}$ by optimizing for one edge parameter $\theta_{k;e}$ at a time by 
 solving 
 \begin{align}\label{eq:alg_high_level}
 	\htheta_{k+1;e} = \argmax_{\htheta\in [-1,1]} \left[ \overline{f}_{k;e}(\htheta):=  \frac{1}{m} \sum_{i=1}^{m} \ell( \htheta_{k+1;1},\dots,\htheta_{k+1;e-1},\htheta,\htheta_{k;e+1},\dots,\htheta_{k;|E|} ;\sigma^{(i)}),  \right]
 \end{align}
 assuming that we label the edge set $E$ as integers from 1 through $|E|$. This can be done by solving for  
\begin{align}\label{eq:fixed_point_eq}
 	\frac{\partial}{\partial\htheta_e} \overline{f}_{k;e}(\htheta) = 0,
 \end{align}
if such a point exists
(using, e.g., \cite{brent2013algorithms}). 

To analyze such an algorithm, one needs to consider the partial derivatives of the log-likelihood function in \eqref{eq:def_log_likelihood}. 
\begin{lemma}[Likelihood and magnetization]\label{lem:derivative}
	Let $\sigma$ be a single sample from the true model $\PR_{\param^{*}}$ over $T$ and let $e = \{x,y\}$ be an edge in the tree $T$. We let $Z_x := Z_x^{\hparam, T_x} := Z_x^{\hparam, T_x}(\sigma_{L_x})$ (resp.~$Z_y := Z_y^{\hparam, T_y} := Z_y^{\hparam, T_y}(\sigma_{L_y})$) be the magnetization at the root $x$ (resp.~$y$) of the descendant subtree $T_x$ (resp.~$T_y$) with respect to $y$ (resp.~$x$).
Then we have 
		\begin{align}
			\frac{\partial}{\partial\htheta_e}\ell(\hparam;\sigma) &= \frac{Z_x Z_y}{1+Z_x Z_y\htheta_e}.
   \label{eqn:derivative}
		\end{align}
\end{lemma}

Therefore, the function on the left-hand side of the fixed point equation \eqref{eq:fixed_point_eq} is the empirical mean of random rational functions in the edge parameter $\htheta_{e}$, where the coefficients are given by the magnetizations at the end vertices of the edge $e$, computed using the current estimated parameter $\hparam$ with independent leaf configurations $\sigma^{(i)}$. In Theorem \ref{thm:gradient} below, we establish an approximate expression for the partial derivatives of the empirical log-likelihood function, in the population limit $m\rightarrow\infty$, 
up to an error of order $O(\delta)$. 

	\begin{theorem}[Population log-likelihood landscape: gradient]\label{thm:gradient}
		There exists $C_{\ref{eqn:gradient}} = C_{\ref{eqn:gradient}}(c_{\ref{eqn:pBounds}},C_{\ref{eqn:pBounds}}, c_{\ref{eqn:pHatBounds}},C_{\ref{eqn:pHatBounds}})>0$ and $\delta_{\ref{eqn:gradient}}=\delta_{\ref{eqn:gradient}}(c_{\ref{eqn:pBounds}},C_{\ref{eqn:pBounds}}, c_{\ref{eqn:pHatBounds}},C_{\ref{eqn:pHatBounds}})$ such that the following holds for any binary tree $T$ and any $\delta<\delta_{\ref{eqn:gradient}}$. If \ref{assumption1} holds for that $\delta$ then for every $e\in E(T)$,
		\begin{equation}\label{eqn:gradient}
			\left|\frac{\partial}{\partial \htheta_e} \E_{\param^{*}} \left[ \ell(\hparam;\sigma) \right]-\frac{\theta^*_e-\htheta_e}{1-\htheta_e^{2} } \right| \le C_{\ref{eqn:gradient}}\delta.
		\end{equation}
	\end{theorem}
The proof of the result above is provided in Section \ref{sec:keyLemmas}, where our main ancestral reconstruction result (Thm. \ref{thm:Robust}) is used crucially. 
We give a brief sketch of the main idea.
By Theorem \ref{thm:Robust}, we have $Z_{x}Z_{y}\approx \sigma_{x}\sigma_{y}$ with probability close to $1$, so essentially we have 
	\begin{align}
		\nonumber\frac{\partial}{\partial \htheta_e} \E_{\param^{*}} \left[ \ell(\hparam;\sigma) \right] &= \E_{\param^{*}}\left[\frac{\partial}{\partial\htheta_e}\ell(\hparam;\sigma)\right]\\ &\approx \nonumber\E_{\param^{*}}\left[\frac{\sigma_x\sigma_y}{1+\sigma_x\sigma_y \htheta_e}\right]\\ &= \nonumber \frac{1}{1+\htheta_e}\PR_{\param^{*}}(\sigma_x=\sigma_y) - \frac{1}{1-\htheta_e}\PR_{\param^{*}}(\sigma_x\neq\sigma_y)\\
			 &= \frac{1}{2}\left(\frac{1+\theta^*_e}{1+\htheta_e}-\frac{1-\theta^*_e}{1-\htheta_e} \right)\\ 
   &= \frac{\theta^*_e-\htheta_e}{1-\htheta_e^2},\label{eqn:grad00}
	\end{align} 
where $\sigma$ is a single sample from the true model $\PR_{\param^{*}}$ over $T$.
The main difficulty is to quantify the error of the approximation in the second step in the above display. The issue is that the (random) gradient of the log-likelihood $\frac{Z_xZ_y}{1+Z_xZ_y\htheta_e}$ in \eqref{eqn:derivative} can blow up (i.e., is of order $\delta^{-1}$) with some small probability 
	when $Z_xZ_y\approx -1$. In order to control the expectation, we need to have tight control of the probabilities with which the denominator gets close to zero. To spell out further details, we note that the blow-up can happen in the following two ways: 
	(1) $\sigma_x Z_x, \sigma_y Z_y\approx 1$ but $\sigma_x\sigma_y = -1$; or
	(2) $\sigma_x\sigma_y = 1$, but either $\sigma_x Z_x\approx -1$, $\sigma_yZ_y \approx 1$ or $\sigma_x Z_x\approx 1, \sigma_yZ_y\approx -1$. 
	The former is already accounted for in the above heuristic. The latter is rare; it has probability $O(\delta^2)$ by the reconstruction theorem. Therefore the expected contribution from the event $\sigma_x\sigma_y = 1$ but $Z_xZ_y\approx -1$ ends up being $O(\delta)$.

 

    
Theorem \ref{thm:gradient} gives some insights into why coordinate maximization may work for branch length estimation. For large enough sample size $m$, the partial derivative of the empirical log-likelihood with respect to each branch length parameter $\htheta_{e}$ should be close to its population expectation, which almost vanishes uniquely at the true branch length $\theta_{e}^{*}$. More precisely, Theorem \ref{thm:gradient} shows that the derivative in the population limit will vanish at a branch length $\htheta_{e}$ for which $\frac{\theta^*_e-\htheta_e}{1-\htheta_e^{2} }$ is of order $O(\delta)$. Note that the denominator is of order $O(\delta)$, so this requires the numerator to be at most of order $O(\delta^{2})$. Thus, solving the fixed point equation \eqref{eq:fixed_point_eq} in the population limit \textit{only once} yields a $O(\delta^{2})$-accurate branch length estimate. This observation is stated in Corollary \ref{cor:coord_ini} below. 

	\begin{corollary}[Initialization by coordinate maximization]\label{cor:coord_ini}
		Suppose the true parameter $\param^{*}$ and the initial estimate $\hat{\param}_{0}$ satisfies \ref{assumption1}. Let $\hat{\param}_{1}$ be obtained by a single round of coordinate maximization. Then there exists a constant $C_{\ref{eqn:coord_ini_guarantee}}>0$ such that,  in the population limit $m\rightarrow\infty$, 
		\begin{align}\label{eqn:coord_ini_guarantee}
			\lVert \param^{*} - \hparam_{1} \rVert_{\infty} \le C_{\ref{eqn:coord_ini_guarantee}} \delta^{2}. 
		\end{align}
	\end{corollary}


    While Corollary \ref{cor:coord_ini} shows why coordinate maximization in the empirical likelihood landscape may succeed in obtaining an accurate estimation of the true parameter $\param^{*}$, the result itself does not hold much of a practical significance because it only holds in the population limit. 
    In a companion paper, we establish that coordinate maximization is indeed guaranteed to return an accurate estimation of the true parameter with high probability in the finite sample regime, where the required sample size grows polynomially in the size of the tree (when sufficiently ``balanced''). This requires a substantial analysis of the Hessian of the log-likelihood function---an analysis where our main result (Thm.~\ref{thm:Robust}) plays a key role again. 

	\section{Proof of main result (Theorem \ref{thm:Robust})}\label{sec:MAgLemmaProofs}

	In this section, we prove Theorem \ref{thm:Robust}. We define $q(x,y) := \frac{x+y}{1+xy}$ so that \eqref{eqn:recursionBorg} now reads as 
\begin{equation}\label{eqn:recursionBorg2}
	Z_u = q(\hat{\theta}_{v}Z_v, \hat{\theta}_w Z_{w}).
\end{equation}
	We will use the following useful facts frequently:
		$(s,t)\mapsto q(s,t)$ is increasing in both $s,t\in(-1,1)$ as $\frac{\partial}{\partial s} q(s,t) = \frac{1-t^2}{(1+st)^2}>0$ for $s,t\in(-1,1)$, and similarly for $t$.

	\subsection{Conditional independence of magnetizations} \label{sec:MagLemmas}

	We will need an important property of the magnetization at distinct nodes: they are conditionally independent given the spins at intermediate nodes. More precisely, suppose we have two node-disjoint descendant subtrees $T_{u}$ and $T_{v}$. Then for any node $w$ along the shortest path between $u$ and $v$, 
		\begin{align}\label{eq:Z_conditional_independence}
			Z_{u} \Perp   Z_{v} \,|\, \sigma_{w}, 
		\end{align}
		which follows from the Markov property of the CFN model and the fact that $Z_{u}$ is determined by $\sigma_{L_{u}}$. In fact, the ``unsigned magnetizations'' $\sigma_{u}	Z_{u}$ are independent as long as the supporting descendant subtrees are node-disjoint, as stated in Claim \ref{lem:signed_mag} below. We state this in a slightly more general formulation than what we need.
	
	\begin{claim}[Independence of unsigned magnetizations]\label{lem:signed_mag}
		The following hold:
		\begin{description}
			\item[(i)] Let $u$ be a node in $T$ with a descendant subtree $T_{u}$ and corresponding magnetization $Z_{u}$. Then the unsigned magnetization $\sigma_{u}Z_{u}$ is independent of $\sigma_{u}$ under $\PR_{\param^{*}}$. 

            \vspace{0.1cm}
			\item[(ii)] Let $v_{1},\dots,v_{k}$ be nodes in $T$ and suppose there are corresponding descendant subtrees $T_{v_{1}},\dots,T_{v_{k}}$ that are node-disjoint. Let $Z_{v_{i}}$ denote the corresponding magnetization at $v_{i}$ for $i=1,\dots, k$.
			Then the $\sigma_{v_{i}} Z_{v_{i}}$s, $i=1,\dots,k$, are independent under $\PR_{\param^{*}}$. 
		\end{description}
	\end{claim}
	
		\begin{proof}
		We will first show \textbf{(i)}. Denote $\zeta_{u}:=\sigma_{u}Z_{u}$.
		It is enough to show that 
		\begin{align}\label{eq:Z_invariance_symmetry}
			\zeta_{u} \,|\, \{\sigma_{u}=1\} \,\, \overset{d}{=} \,\,	\zeta_{u}\,|\, \{\sigma_{u}=-1\} \,\, \overset{d}{=} \,\, \zeta_{u}.
		\end{align}
		To establish this, note that, for any node set $V^{(i)} \subseteq V(T_{u})$, by symmetry of the model
		\begin{align}\label{eq:spin_sign_flipping}
			\sigma_{V^{(i)}} \,|\, \{\sigma_{u}=1\}\,\, \overset{d}{=} \,\, -	\sigma_{V^{(i)}} \,|\, \{\sigma_{u}=-1\}. 
		\end{align}
		By applying \eqref{eq:spin_sign_flipping} with $V^{(i)}=L_{u}$ (the set of leaves in the descendant subtree $T_{u}$), we get 
		\begin{align}\label{eq:spin_sign_flipping2}
			Z_{u}(\sigma_{L_{u}}) \,|\, \{\sigma_{u}=1\} \,\, \overset{d}{=} \,\,	Z_{u}(-\sigma_{L_{u}}) \,|\, \{\sigma_{u}=-1\}  =  -Z_{u}(\sigma_{L_{u}}) \,|\, \{\sigma_{u}=-1\},
		\end{align}
		where the last equality follows from the definition of the magnetization. 
Recalling that $\zeta_{u}=	\sigma_{u} Z_{u}(\sigma_{L_{u}})$, we deduce
		\begin{align}\label{eq:Z_invariance_symmetry2}
			\zeta_{u} \,|\, \{\sigma_{u}=1\} \,\, \overset{d}{=} \,\,	\zeta_{u}\,|\, \{\sigma_{u}=-1\}.
		\end{align}
  Then, for any measurable subset $A\subseteq \R$, using \eqref{eq:Z_invariance_symmetry2}, 
		\begin{align*}
			\PR_{\param^{*}}(\zeta_{u}\in A) &=  \sum_{s\in \{\pm 1\}} \PR_{\param^{*}}(\zeta_{u}\in A\,|\, \sigma_{u} = s) \, \PR_{\param^{*}}(\sigma_{u}=s)  = \PR_{\param^{*}}(\zeta_{u}\in A\,|\, \sigma_{u} = +1).
		\end{align*}
		It follows that $\zeta_{u}$ is independent of $\sigma_{u}$, which establishes \eqref{eq:Z_invariance_symmetry}. This shows \textbf{(i)}. 
		
		Next, we show \textbf{(ii)}. Denote $\zeta_{i}:=\sigma_{v_{i}}Z_{v_{i}}$ for $i=1,\dots,k$. Denote $R:=\{v_{1},\dots,v_{k}\}$. 
		Note that by the conditional independence property of the spin configuration $\sigma$ and magnetizations (see \eqref{eq:Z_conditional_independence}), it follows that the unsigned magnetizations $\zeta_{i}$ are conditionally independent given $\sigma_{R}$. Let $A_{1},\dots,A_{k}$ be  measurable subsets of $\R$. Then 
		we can conclude 
		\begin{align}\nonumber
			\PR_{\param^{*}}(\zeta_{1}\in A_{1},\dots, \zeta_{k}\in A_{k})	&= \sum_{\tau_{R}:R\rightarrow \{\pm 1\}}  \PR_{\param^{*}}(\zeta_{1}\in A_{1},\dots, \zeta_{k}\in A_{k}\,|\, \sigma_{R}=\tau_{R})\,\PR_{\param^{*}}(\sigma_{L}=\tau_{L}) \\
\nonumber			&=\sum_{\tau_{R}:R\rightarrow \{\pm 1\}}   \prod_{i=1}^{k} \PR_{\param^{*}}(\zeta_{i}\in A_{i}\,|\, \sigma_{v_{i}}=\tau_{v_{i}}) \, \PR_{\param^{*}}(\sigma_{R}=\tau_{R}) \\
	\nonumber		&= \prod_{i=1}^{k} \PR_{\param^{*}}(\zeta_{i}\in A_{i})  \sum_{\tau_{R}:R\rightarrow \{\pm 1\}}    \PR_{\param^{*}}(\sigma_{R}=\tau_{R}) \\
	\nonumber		&= \prod_{i=1}^{k} \PR_{\param^{*}}(\zeta_{i}\in A_{i}), 
		\end{align}
		where the second equality is due to conditional independence and the following equality is from 	
		\eqref{eq:Z_invariance_symmetry}. 
	\end{proof}

	\subsection{Bounds on $q$} We first establish a few useful bounds on the function $q$. 
	The next claim bounds the output of $q$ when both incoming inputs are ``strong,'' i.e., both are either close to $1$ or $-1$, and agree.
	\begin{claim}[Magnetization: two strong agreeing inputs]\label{claim:quickBound1}
		Fix any $\eps\in [0,\frac{1}{2})$. If $s,t$ are such that $1-\eps\le s,t\le 1$ then
		\begin{equation*}
			q(s,t) \ge 1-\frac{4}{5}\eps^2\qquad \textup{ and }\qquad q(-s,-t)\le -1+\frac{4}{5}\eps^2.
		\end{equation*}
	\end{claim}
	\begin{proof} Since $st\ge \frac{1}{4}$, 
		\begin{align*}
			1-q(s,t) = \frac{1+st - s-t}{1+st} = \frac{(1-s)(1-t)}{1+st} \le \frac{\eps^2}{1+st} \le \frac{4}{5} \eps^2.
		\end{align*}
		Note that $q(-s,-t) = -q(s,t)$ gives the other claim.
	\end{proof}

We will also need the following observation regarding the function $q$ in the ``opposite case,'' i.e., when both inputs are ``strong'' but \emph{in the opposite directions}.
 \begin{claim}[Magnetization: two strong disagreeing inputs]\label{claim:strongopposites.}
     Fix any $0<a<A$ and let $\delta>0$ with $A\delta<1$. For any $s,t \in [1-A\delta, 1-a\delta]$ then
     \begin{equation*}
        -1+\frac{a}{A}\le q(s,-t) = q(-s,t) \le 1-\frac{a}{A}.
     \end{equation*}
     More generally, if $t\in [1-A\delta, 1-a\delta]$ and $|s|\le 1-a\delta$ then
     \begin{equation*}
         q(s,t) \ge -1 + \frac{a}{A}.
     \end{equation*}
 \end{claim}
 \begin{proof}
By using the monotonicity of $q(\cdot,\cdot)$ in both arguments, 
     \begin{align*}
         q(s,-t) \ge q (1-A\delta, -1+a\delta) = \frac{(a-A)\delta}{(A+a)\delta- aA \delta^2} = - \frac{A-a}{A+a- aA\delta} \ge - \frac{A-a}{A}. 
     \end{align*} In the last inequality we use $A+a - aA\delta \ge A + a-a$ since $A\delta <1$. 
     For the opposite bound, it is simply
     \begin{equation*}
         q(s,-t) \le q (1-a\delta, -1+A\delta) = - q(1-A\delta, -1+a\delta).
     \end{equation*}
     \end{proof}

	\subsection{Effect of recursion} Our proof of Theorem \ref{thm:Robust}
	will be based on a case analysis for a three-level subtree rooted at the target vertex. 
	The next claim, which will be central to the analysis, bounds the effect
	of a ``corruption'' (i.e., a flip or failure of reconstruction) followed by three recursions of $q$ (with ``strong signals''). Below, $s_1$ corresponds to a possible corruption.
	
	\begin{claim}[Magnetization recursion: corruption at distance $3$]\label{claim:quickBound2}
		Suppose that $0<a<A$, $0<B$ and $\delta\in [0,1]$ are such that $a<A/2$, $\delta<a/2$, 
        and  $\left(\frac{2 A^2}{a} + B\right)\delta<1/2$.
        If $s_1 \in [-1+a\delta,1]$ and, 
		for each $j\in\{2,3,4\}$ and $i\in\{1,2\}$, $
			s_j\in [1-A\delta,1]$ {and} $t_i\in[1-B\delta,1]$, then we have 
  \begin{align}
  t_2 q(t_1q(s_1,s_2), s_3)
		\label{eqn:poly0} &\geq 1 - \left(\frac{2 A^2}{a} + B\right)\delta,\\
			q\Bigg( t_2q\Big(t_1q(s_1,s_2)  , s_3\Big), s_4  \Bigg) &\ge 1-\frac{4}{5}\left(\frac{2 A^2}{a} + B\right)\delta^2
			\label{eqn:poly1}.
		\end{align}
	\end{claim}
	
	\begin{proof}
		We bound the composition of functions from the inside one term at a time. 
  \begin{enumerate}
  \item First, as $0 < a \delta < A \delta < 1$, $s_1 \in [-1+a\delta,1]$ and $s_2\in [1-A\delta,1]$,
		\begin{equation*}
			q(s_1,s_2)\ge q(-1+a\delta,1-A\delta) = \frac{a-A}{a+A-aA \delta},
		\end{equation*}
  by the monotonicity of $q(\cdot, \cdot)$ in both arguments.
		Note that the rightmost expression is always negative (and $> -1$) by our assumptions. Since $0<t_{1}\le 1$, it follows that 
  $$t_1\, q(s_1,s_2) \ge \frac{a-A}{a+A-aA \delta}.$$ 
  
  \item By monotonicity again, we get using the previous bullet that
		\begin{align*}
			q(t_1q(s_1,s_2), s_3)
			 &\ge q\left(\frac{a-A}{a+A-aA\delta}, 1-A\delta \right)\\  &= \frac{\frac{a-A}{a+A-aA\delta} + 1-A\delta}{1 + 
				\frac{a-A}{a+A-aA\delta}(1-A\delta)}\\
			&= \frac{2a + \left(-2a-A\right)A\delta + aA^{2}\delta^{2}}{2a + \left(-2a+A \right)A\delta 
            }\\ &= 1 - \frac{2A^2 \delta + aA^{2}\delta^{2}}{2a + \left(-2a+A \right)A\delta } \\
            &\overset{(a)}{\ge}  1 - \frac{2A^2 \delta + aA^{2}\delta^{2}}{2a}\\ &= 1 -  \frac{A^2}{a}\delta - \frac{A^{2}\delta^{2}}{2}\\
            &\overset{(b)}{\ge} 1 - \frac{2 A^2}{a}\delta,
		\end{align*} 
        where (a) uses the hypothesis $a<A/2$ and (b) uses $\delta<a/2$. Since $t_2 \geq 1 - B \delta$,
		from the lower bound on $q(t_1q(s_1,s_2), s_3)$ above, equation \eqref{eqn:poly0} follows.
  
		\item Finally, using the last display and $s_4 \geq 1 - A\delta\ge 1 - \left(\frac{2 A^2}{a} + B\right)\delta$ (since $0<a<A$ and $B>0$), Claim~\ref{claim:quickBound1} with $\eps =\left(\frac{2 A^2}{a} + B\right)\delta<1/2$ (by the hypothesis) yields \eqref{eqn:poly1}.
  \end{enumerate}
  That concludes the proof.
	\end{proof}

 \subsection{Base case: three-level subtree}
	We now use Claims \ref{claim:quickBound1} and \ref{claim:quickBound2}
	to analyze the main subcase of
	the proof of Theorem \ref{thm:Robust}.
	We assume that vertex $u$ has $8$
	great-grandchildren in $T_u$, that is,
	that the tree looks like this:
	\begin{equation}\label{Tree:T_u}
		T_u = \text{\small   {\qroofx=1
				\qroofy=1  
				\Tree [.$u$ [.$v_1$ [.$w_1$ \qroof{$\qquad$}.$x_1$  \qroof{$\qquad$}.$x_2$ ] [.$w_2$ \qroof{$\qquad$}.$x_3$ \qroof{$\qquad$}.$x_4$ ] ] [.$v_2$ [.$w_3$ \qroof{$\qquad$}.$x_5$ \qroof{$\qquad$}.$x_6$ ] [.$w_4$ \qroof{$\qquad$}.$x_7$ \qroof{$\qquad$}.$x_8$ ] ] ]}}
	\end{equation} 
	We will address the more general situation in the proof of Theorem \ref{thm:Robust} below.
	Also, in the proof of that theorem, we will eventually need to show that the constant in \eqref{eqn:magWorks2} below ``recurses through,'' and therefore we take some care in tracking constant factors.
	\begin{claim}[Three-level subtree: single leaf corruption]\label{lem:Case1}
		Assume that $\hparam\in \hParam_0(\delta)$ in \eqref{eqn:pHatBounds}. Fix $\sigma_{T_u}\in \{\pm 1\}$ and assume the following hold:
        \begin{description}

            \item[(i)] (\textit{No flip except possibly on $\{w_{1},x_{1}\}$}) $\sigma_y = \sigma_u$ for all $y\in \{v_1,v_2, w_1,\dotsm, w_4, x_2,x_3,\dotsm, x_8\}$.       
            \vspace{0.1cm}
            \item[(ii)] (\textit{Strong signal at remaining leaves}) There is a positive constant $C_{\ref{eqn:magWorks2}}$ such that for all $j=2,\dotsm, 8$,  
		\begin{equation}\label{eqn:magWorks2}
			\sigma_{x_j} Z_{x_j}^{\hparam}(\sigma_{L_{x_j}}) \ge 1-C_{\ref{eqn:magWorks2}}\delta^2.
		\end{equation} 
            \end{description}

        Define $C_{\ref{eqn:3lev1cor}}
		:= \frac{4}{5}\left(9\frac{C_{\ref{eqn:pHatBounds}}^2}{c_{\ref{eqn:pHatBounds}}} + 2 C_{\ref{eqn:pHatBounds}}\right)^2
		$. It holds that  
		\begin{equation}\label{eqn:3lev1cor}
			\sigma_u Z_u^{\hparam}(\sigma_{L_u}) \ge 1-C_{\ref{eqn:3lev1cor}} \delta^2
		\end{equation}
  provided that $\delta\le \delta_{\ref{eqn:3lev1cor}} := \delta_{\ref{eqn:3lev1cor}}( c_{\ref{eqn:pHatBounds}}, C_{\ref{eqn:pHatBounds}}, C_{\ref{eqn:magWorks2}}) :=\frac{C_{\ref{eqn:pHatBounds}}}{2 C_{\ref{eqn:magWorks2}}} \wedge \frac{5}{72 C_{\ref{eqn:pHatBounds}}}\wedge c_{\ref{eqn:pHatBounds}} > 0$.
		
	\end{claim}
	\begin{proof}
		Without loss of generality, we suppose that $\sigma_u = 1$.
		As before, we simplify the notation
		by writing $Z_y = Z_y^\hparam(\sigma_{L_y})$ for all vertices $y$. The only possible ``corruption'' occurs at $x_1$. Our goal is to apply Claim \ref{claim:quickBound2} on the path between $u$ and $x_1$. We first need to control the inputs off the path.
		
		\medskip\noindent\textbf{Magnetization at the vertices off the path} We apply the magnetization recursion \eqref{eqn:recursionBorg2} and 
  Claim \ref{claim:quickBound1} to all vertices \emph{off the path} between
		$u$ and $x_1$, i.e., $v_2$, $w_2$ and $x_2$. The latter is already covered by the assumptions. We detail the case of $w_2$ and $x_2$. 
  \begin{itemize}
  \item First, for $w_2$, we use that $Z_{x_{3}},Z_{x_{4}} \ge 1-C_{\ref{eqn:magWorks2}}\delta^{2}$ under the hypothesis. Note that for any constant $C>0$ and edge $e$, 
        \begin{align}\label{eq:hat_theta_mag_lower_bd}
          \hat{\theta}_{e}  (1-C\delta^{2}) \ge ( 1 - 2C_{\ref{eqn:pHatBounds}}\delta) (1-C\delta^{2}) 
          \ge 1-3C_{\ref{eqn:pHatBounds}}\delta \qquad \textup{for all }\delta \le \delta_{\ref{eq:hat_theta_mag_lower_bd}}(C_{\ref{eqn:pHatBounds}},C):=\frac{C_{\ref{eqn:pHatBounds}}}{2C} 
        \end{align}
        where the first inequality uses $\htheta_{e} \geq 1 - 2C_{\ref{eqn:pHatBounds}}\delta$ in \ref{assumption1}.
        We claim that 
\begin{align}\label{eq:three_step_mag_recursion_pf1}
        Z_{w_{2}} \ge  1 - \frac{4}{5}(3C_{\ref{eqn:pHatBounds}})^{2}\delta^{2}  
        =
        1 - \frac{36}{5}C_{\ref{eqn:pHatBounds}}^{2}\delta^{2}
        \qquad \textup{for all }\delta\le \delta_{\ref{eq:hat_theta_mag_lower_bd}}(C_{\ref{eqn:pHatBounds}},C_{\ref{eqn:magWorks2}})=\frac{C_{\ref{eqn:pHatBounds}}}{2 C_{\ref{eqn:magWorks2}}}.
        \end{align}
        Indeed, applying Claim \ref{claim:quickBound1} with $s=\hat{\theta}_{\{w_{2},x_{3}\}} Z_{x_{3}}$ and $t=\hat{\theta}_{\{w_{2},x_{4}\}} Z_{x_{4}}$ along with \eqref{eq:hat_theta_mag_lower_bd} with $C=C_{\ref{eqn:magWorks2}}$  
        verifies \eqref{eq:three_step_mag_recursion_pf1}. 

\item    The claim in \eqref{eq:three_step_mag_recursion_pf1} holds for $Z_{w_3}$ and $Z_{w_4}$ as well by an identical argument. Then, applying  Claim \ref{claim:quickBound1} with $s=\hat{\theta}_{\{v_{2},w_{3}\}} Z_{w_{3}}$ and $t=\hat{\theta}_{\{v_{2},w_{4}\}} Z_{w_{4}}$ along with \eqref{eq:hat_theta_mag_lower_bd} with $C=\frac{36}{5} C_{\ref{eqn:pHatBounds}}^{2}$, we deduce 
\begin{align}\label{eq:three_step_mag_recursion_pf2}
        Z_{v_{2}} \ge  1 - \frac{36}{5}C_{\ref{eqn:pHatBounds}}^{2}\delta^{2}\qquad \textup{for all }\delta\le \delta_{\ref{eq:three_step_mag_recursion_pf2}}(C_{\ref{eqn:pHatBounds}},C_{\ref{eqn:magWorks2}}):=\frac{C_{\ref{eqn:pHatBounds}}}{2 C_{\ref{eqn:magWorks2}}} \wedge \frac{5}{72 C_{\ref{eqn:pHatBounds}}}.
        \end{align}
        By definition, $\delta_{\ref{eq:three_step_mag_recursion_pf2}}(C_{\ref{eqn:pHatBounds}},C_{\ref{eqn:magWorks2}}) \leq \delta_{\ref{eq:hat_theta_mag_lower_bd}}(C_{\ref{eqn:pHatBounds}},C_{\ref{eqn:magWorks2}})$. Hence \eqref{eq:three_step_mag_recursion_pf2} holds for $Z_{w_2}$ as well.
    \end{itemize}

        \medskip\noindent\textbf{Magnetization at the root} Having a lower bound on the magnetization at the nodes $w_{2}$ and $v_{2}$, we proceed similarly to lower bound $Z_{u}$, this time using Claim \ref{claim:quickBound2}. Namely, we take $s_4 = \htheta_{\{u,v_2\}} Z_{v_2}$, $s_3 = \htheta_{\{v_1,w_2\}} Z_{w_2}$, $s_2 = \htheta_{\{w_1,x_2\}} Z_{x_2}$, $s_1 = \htheta_{\{w_1,x_1\}} Z_{x_1}$, $t_2 = \htheta_{\{u,v_1\}}$
		and $t_1 = \htheta_{\{v_1,w_1\}}$.
		In all cases, for $\delta\le \delta_{\ref{eq:three_step_mag_recursion_pf2}}(C_{\ref{eqn:pHatBounds}},C_{\ref{eqn:magWorks2}})$,
		\begin{align*}
			&-1 + 2c_{\ref{eqn:pHatBounds}}\delta \leq s_1 \leq 1\qquad \textup{by \ref{assumption1} and $Z_{x_1} \in [-1,1]$}\\
			&1 - 3 C_{\ref{eqn:pHatBounds}} \delta \leq s_2, s_3, s_4 \leq 1 \qquad \textup{by \eqref{eq:hat_theta_mag_lower_bd}, \eqref{eq:three_step_mag_recursion_pf1}, and \eqref{eq:three_step_mag_recursion_pf2}} \\
			&1 - 2C_{\ref{eqn:pHatBounds}}\delta \leq t_1, t_2 \leq 1, \qquad \textup{by \ref{assumption1}}.
		\end{align*}
        Then Claim \ref{claim:quickBound2} with $a=2c_{\ref{eqn:pHatBounds}}$, $A=3 C_{\ref{eqn:pHatBounds}}$, $B=2C_{\ref{eqn:pHatBounds}}$ yields 
        \begin{align}\nonumber
        Z_{u} \ge 1 -  \frac{4}{5} \left(\frac{9 C_{\ref{eqn:pHatBounds}}}{c_{\ref{eqn:pHatBounds}}} + 2C_{\ref{eqn:pHatBounds}}\right)^2 \delta^2,
        \end{align}
        which holds whenever $\delta\le \delta_{\ref{eq:three_step_mag_recursion_pf2}}(C_{\ref{eqn:pHatBounds}},C_{\ref{eqn:magWorks2}})\land c_{\ref{eqn:pHatBounds}}$ (where we used the condition $\delta < a/2$). \end{proof}

	\subsection{Proof of Theorem \ref{thm:Robust}}
	
	We are now ready to prove Theorem \ref{thm:Robust}. As we mentioned previously, we recurse on a three-level subtree. The choice of \emph{three} levels is driven by the fact that applying the $q$ function that many times ultimately erases the effect of ``corruption'' (provided other input are ``strong signals''), as seen in Claim~\ref{claim:quickBound2}.
	
	\begin{proof}[\textbf{Proof of Theorem} \ref{thm:Robust} \textbf{\textup{(i)}}] 
		We start with the case where
		$u$ has $8$ great-grandchildren in $T_u$ and use the notation of
		\eqref{Tree:T_u}. 
		That is, we call $v_1,v_2$ the vertices on level 1,  $w_1,\dotsm,w_4$ the vertices on level 2, and  $x_1,\dotsm, x_8$ the vertices on level 3.
		(We come back to the more general
		situation at the end of the proof.)
		
		We work by induction. Formally,
		we assume that for all $j\in[8]$,
		\begin{equation*}
			\PR_{\param^{*}}(\sigma_{x_j} Z_{x_j} \ge 1-C_{\ref{eqn:Reconstruct}}\delta^2) \ge 1-c_{\ref{eqn:Reconstruct}}\delta,
		\end{equation*}
        where the constants $c_{\ref{eqn:Reconstruct}}$ and $C_{\ref{eqn:Reconstruct}}\delta^2$ are defined in the statement. 
		Our goal is to establish 
        \begin{align}\label{eq:recons_lemma_pf1}
			\PR_{\param^{*}}(\sigma_{u} Z_{u} \ge 1-C_{\ref{eqn:Reconstruct}}\delta^2) \ge 1-c_{\ref{eqn:Reconstruct}}\delta.
		\end{align}
		
		We say that $y$ is a \textit{mutation} on level $k$ if $\sigma_y \neq \sigma_{y'}$ for its parent $y'$ on level $k-1$.
		We define the following events
		\begin{align*}
			A_j &:= \{\sigma_{x_j} Z_{x_j}< 1-4C_{\ref{eqn:pHatBounds}}\delta^2\} \quad \textup{for $j\in[8]$},\\
			B_{T} &:=  \{\textup{there is at least 1 mutation on levels 1 and/or 2}\}, \\
			B_L &:= \{ \textup{there is at least 2 mutations on level 3}\},\\
			B_1&:= \{\textup{there is at least 1 mutation one level 3}\} \cap \bigcup_{i=1}^8 A_i, \\
			B_{2}&:= \bigcup_{i< j} (A_i\cap A_j),\\
			E&:= B_T\cup B_L\cup B_1\cup B_2.
		\end{align*}
		We claim that on $E^c$, $\sigma_uZ_u\ge 1-C_{\ref{eqn:Reconstruct}}\delta^2$. To see this, note that in the event $E^c$, at least one of the following three things must hold:
		\begin{enumerate}
			\item there is no mutation on any level, and none of the events $A_i$ occurs; or
			\item there is only 1 mutation, which occurs on the third level, and none of the events $A_i$ occurs; or
			\item there is no mutation on any level and precisely one $A_i$ occurs.
		\end{enumerate}
		In all these cases, $A_{j}^c$ occurs for at least 7 distinct $j$'s. Hence by Claim \ref{lem:Case1} with $C_{\ref{eqn:magWorks2}}=4C_{\ref{eqn:pHatBounds}}$, we have 
        \begin{align}\nonumber
        \sigma_u Z_u \ge 1-C_{\ref{eqn:Reconstruct}} \delta^2 \qquad \textup{for all $\delta\le \delta_{\ref{eqn:3lev1cor}}(c_{\ref{eqn:pHatBounds}}, C_{\ref{eqn:pHatBounds}}, 4C_{\ref{eqn:pHatBounds}})$},
        \end{align}
        noting that the constant $C_{\ref{eqn:Reconstruct}}$ defined in the statement is exactly $C_{\ref{eqn:3lev1cor}}
		= \frac{4}{5} \left(\frac{9 C_{\ref{eqn:pHatBounds}}^2}{c_{\ref{eqn:pHatBounds}}} + 2 C_{\ref{eqn:pHatBounds}}\right)^2
		$ from Claim \ref{lem:Case1}.

		In order to deduce \eqref{eq:recons_lemma_pf1}, it remains to show that 
        \begin{align*}
            \PR_{\param^{*}}(E) \le c_{\ref{eqn:Reconstruct}} \delta,
        \end{align*}
        for $\delta$ small enough. We separately bound the ``bad events.''
\begin{itemize}
        \item First, observe that
		\begin{align*}
			\PR_{\param^{*}}(B_T) &= 1- \PR_{\param^{*}}(B_T^c) \le  1-\left(1-C_{\ref{eqn:pBounds}}\delta\right)^6 \le 6C_{\ref{eqn:pBounds}}\delta,
		\end{align*}
		since $B_T^c$ is the event that no mutation occurs amongst 6 independent possibilities (under the true model). 
		
  \item Next
		\begin{equation*}
			\PR_{\param^{*}}(B_L) \le \binom{8}{2} (2C_{\ref{eqn:pBounds}}\delta)^2 = 112 C_{\ref{eqn:pBounds}}^2\delta^2.
		\end{equation*}
		
  \item Also 
		\begin{align*}
			\PR_{\param^{*}}(B_1) 
			&= \PR_{\param^{*}}(\{\textup{there is at least 1 mutation on level 3}\})
			\times \PR_{\param^{*}}(\cup_j A_j)\\
			&\le \PR_{\param^{*}}(\{\textup{there is at least 1 mutation on level 3}\}) \times 
			8 \max_j\PR_{\param^{*}}(A_j)\\
			&\le 8(2C_{\ref{eqn:pBounds}}\delta)
			\times 8 (c_{\ref{eqn:Reconstruct}} \delta)= 128 c_{\ref{eqn:Reconstruct}}C_{\ref{eqn:pBounds}}\delta^2.
		\end{align*}
		In the first line, we used the fact that $\sigma_{x_j}Z_{x_j}$ is independent of the mutations on level $3$ by Claim \ref{lem:signed_mag} and the Markov property of the CFN model.
		
  \item Similarly, by Claim \ref{lem:signed_mag}, 
		\begin{align*}
			\PR_{\param^{*}}(B_2)\le\binom{8}{2} \max_{i<j}\PR_{\param^{*}}(A_i\cap A_j) = \binom{8}{2} \max_{i<j}\PR_{\param^{*}}(A_i)\PR_{\param^{*}} (A_j) \le 28 c_{\ref{eqn:Reconstruct}}^2 \delta^2.
		\end{align*}
\end{itemize}
Hence a union bound implies 
		\begin{equation*}
			\PR_{\param^{*}}(E) \le 6 C_{\ref{eqn:pBounds}}\delta+\left(112C_{\ref{eqn:pBounds}}^2 + 128 c_{\ref{eqn:Reconstruct}}C_{\ref{eqn:pBounds}}+28 c_{\ref{eqn:Reconstruct}}^2\right)\delta^2  
			\le 7 C_{\ref{eqn:pBounds}}\delta,
		\end{equation*}
        where the last inequality holds whenever (recall that $
        c_{\ref{eqn:Reconstruct}}
		:= 7 C_{\ref{eqn:pBounds}}$)
        \begin{align}\nonumber 
        \delta \le \frac{1}{112C_{\ref{eqn:pBounds}} + 128\times 7 C_{\ref{eqn:pBounds}}+28 \times 7^{2}C_{\ref{eqn:pBounds}}}  = \frac{1}{2380C_{\ref{eqn:pBounds}}}.
        \end{align}

		
		There is one last matter 
		to attend to. We have assumed that $u$ has $8$ great-grandchildren. 
		The rest of the cases are established using a coupling argument. We provide details on this coupling in the case that $v_1$ is a leaf, but $v_2$ still has 4 grandchildren. 
		That is, we consider the tree $\tilde{T}_u$ defined as
		\begin{equation*}
			\tilde{T}_u = \text{{\qroofx=1
					\qroofy=1
					\Tree [.$u$ [.$v_1$ ] [.$v_2$ [.$w_3$ \qroof{$\qquad$}.$x_5$ \qroof{$\qquad$}.$x_6$ ] [.$w_4$ \qroof{$\qquad$}.$x_7$ \qroof{$\qquad$}.$x_8$ ] ] ]}}
		\end{equation*} Also let $T_u'$ denote the tree depicted 
		\begin{equation*}
			T_u' = \text{{\qroofx=1
					\qroofy=1
					\Tree [.$u$ [.$v_1$ [.$w_1$ $x_1$  $x_2$ ] [.$w_2$ $x_3$ $x_4$ ] ] [.$v_2$ [.$w_3$ \qroof{$\qquad$}.$x_5$ \qroof{$\qquad$}.$x_6$ ] [.$w_4$ \qroof{$\qquad$}.$x_7$ \qroof{$\qquad$}.$x_8$ ] ] ]}}
		\end{equation*} 
		where $x_1,\dotsm, x_4$ are leaves and the corresponding edges all satisfy \eqref{eqn:pBounds}.
		We denote by $\tilde{Z}$ and $Z'$ the magnetizations on $\tilde{T}_u$ and $T'_u$ respectively. 
		Since 
		\begin{equation*}
			\sigma_{v_1}\tilde{Z}_{v_1} = 1\ge \sigma_{v_1} Z_{v_1}'
		\end{equation*} and $a\mapsto q(a,b)$ is monotonic in $a$, we get the following stochastic domination
		\begin{equation*}
			\sigma_u\tilde{Z}_u \ge_{\textup{st}} \sigma_uZ_u'.
		\end{equation*} 
		If $\tilde{\param},\widetilde{\hparam}$ are the parameters on $\tilde{T}_u$ with ``extensions'' $\param'$, $\hparam'$ on $T'_u$, which each satisfy the corresponding hypotheses \eqref{eqn:pBounds}, \eqref{eqn:pHatBounds}, 
		then we get the implication
		\begin{equation*}
			\PR_{\param'}\left(\sigma_{x_j}Z_{x_j}'\ge 1-C_{\ref{eqn:Reconstruct}}\delta^2\right)\ge 1-c_{\ref{eqn:Reconstruct}}\delta, \, \forall j\in[8] \quad\textup{implies}\quad \PR_{\tilde\param}(\sigma_u Z_u \ge 1-C_{\ref{eqn:Reconstruct}}\delta^2) \ge 1-c_{\ref{eqn:Reconstruct}}\delta.
		\end{equation*}
		Similar coupling arguments hold for the remaining cases.
	\end{proof}

	Next, we turn our attention to proving the lower tail bound in Theorem \ref{thm:Robust} \textbf{(ii)}. 
	Let us look back at the tree $T_u$ in \eqref{Tree:T_u} and the proof of Theorem \ref{thm:Robust} \textbf{(i)}. We can see that if exactly one of the great-grandchildren of $u$ is corrupted by either (a) having the signal $\sigma_u\neq \sigma_{x_j}$ or (b) failing reconstruction in the sense that $\sigma_{x_j} Z_{x_j}\le 1-C_{\ref{eqn:Reconstruct}}\delta^2$, then the magnetization at $u$ is $Z_u\approx \sigma_u$ up to a perturbation by at most $c_{\ref{eqn:Reconstruct}}\delta^2$. The chance that exactly one great-grandchild is corrupted in the prior sense is order $\Theta(\delta)$.
	
	What might result in the magnetization $Z_u\approx -\sigma_u$? The obvious example is where $\sigma_{v_1} = \sigma_{v_2} = -\sigma_u$. This event has probability $\Theta(\delta^2)$, so this suggests that $Z_u\approx -\sigma_u$ may have probability $o(\delta)$. As the proof of Theorem \ref{thm:Robust} \textbf{(ii)} below shows, that is essentially the only way this happens. 
	In its proof we will use Claim \ref{claim:strongopposites.}. 
 
	
	\begin{proof}[\textbf{Proof of Theorem} \ref{thm:Robust} \textup{\textbf{(ii)}}]
		Recall that the statement in (ii) is that $\PR_{\param^{*}} \left(\sigma_u Z_u^{\hparam}(\sigma_{L_{u}}) \le - c_{\ref{eqn:antiReconstruction}}  \right) \le C_{\ref{eqn:antiReconstruction}}\delta^2$ for $\delta\le \delta_{\ref{eqn:antiReconstruction}}$.
		Let us denote the two children of the vertex $u$ by $x$ and $y$. Define the events
		\begin{align*}
			A_x &= \{\sigma_u = \sigma_x\} & A_y&=\{\sigma_u= \sigma_y\}\\
			B_x &= \{\sigma_{x} Z_x \ge 1-C_{\ref{eqn:Reconstruct}} \delta^2\} & B_y &= \{\sigma_y Z_y \ge 1-C_{\ref{eqn:Reconstruct}}\delta^2\},
		\end{align*}
		and observe that these are independent events. Let $E$ be the event that \emph{at least three} of the events $A_x,A_y,B_x,B_y$ occur.
		We claim that we can find a constant $c_{\ref{eqn:antiReconstruction}} \in (0,1)$ such that
		\begin{equation}\label{eqn:eventE} E\subseteq \{\sigma_u Z_u \ge -c_{\ref{eqn:antiReconstruction}}\}.
		\end{equation}
		This will imply the claim. Indeed 
		a union bound gives
		\begin{align*}
			\PR(E^c) &\le P(A_x^c\cap A_y^c) + P(A_x^c\cap B_x^c)+ P(A_x^c\cap B^c_y)\\ &\qquad + P(A_y^c\cap B^c_x)+ P(A^c_y\cap B^c_y) + P(B_x^c\cap B^c_y)\\
			&\le \left(C_{\ref{eqn:pBounds}}^2 + 4 C_{\ref{eqn:pBounds}} c_{\ref{eqn:Reconstruct}}+ c_{\ref{eqn:Reconstruct}}^2\right) \delta^2  = 78 C_{\ref{eqn:pBounds}}^2 \delta^2
		\end{align*}
		where we obtained the second inequality by noting that 
		Assumption \ref{assumption1} implies $\PR(A^c_x), \PR(A^c_y)\le C_{\ref{eqn:pBounds}}\delta$ and that Theorem \ref{thm:Robust} \textbf{(i)} implies $\PR(B^c_x), \PR(B_y^c)\le c_{\ref{eqn:Reconstruct}}\delta$ for $\delta$ small enough.
		
		It remains to prove \eqref{eqn:eventE}. We focus on the event $\{\sigma_u=1\}$, the other case being similar. 
		On $\{\sigma_u=1\} \cap A_x\cap A_y\cap B_x$ we have by \eqref{eq:hat_theta_mag_lower_bd}, 
  \begin{equation*}\htheta_xZ_x \ge (1-2C_{\ref{eqn:pHatBounds}}\delta)(1-C_{\ref{eqn:Reconstruct}}\delta^2) \ge 1-3C_{\ref{eqn:pHatBounds}}\delta \qquad \textup{for all $\delta \le \delta_{\ref{eq:hat_theta_mag_lower_bd}}(C_{\ref{eqn:pHatBounds}},C_{\ref{eqn:Reconstruct}})$}.
		\end{equation*}
		Moreover, 
		as $Z_y\in[-1,1]$ a.s. and $\htheta_y\le 1-2c_{\ref{eqn:pHatBounds}}\delta$, we have also
		\begin{equation*}
			\htheta_yZ_y \ge -1+2c_{\ref{eqn:pHatBounds}}\delta.
		\end{equation*} 
		It can be checked that the same bounds hold on the event $\{\sigma_u=1\} \cap A_x\cap B_x\cap B_y$. 
		Therefore, by coordinatewise monotonicity of $q$, we get
		\begin{align*}
			Z_u = q(\htheta_x Z_x, \htheta_y Z_y) 
			\ge q(1-3C_{\ref{eqn:pHatBounds}}\delta, -1+2c_{\ref{eqn:pHatBounds}}\delta)
			&\overset{(a)}{\ge} -1 + \frac{2 c_{\ref{eqn:pHatBounds}}}{3C_{\ref{eqn:pHatBounds}}} =: -c_{\ref{eqn:antiReconstruction}},
		\end{align*}
        where (a) uses Claim \ref{claim:strongopposites.} with $A=3C_{\ref{eqn:pHatBounds}}$ and $a=2 c_{\ref{eqn:pHatBounds}}$.  
  
		
		A similar argument with upper bounds instead of lower bounds gives $\{\sigma_u=-1\}\cap E\subseteq \{\sigma_u Z_u \ge -c_{\ref{eqn:antiReconstruction}}\}$. 
		That concludes the proof.
	\end{proof}

	\section{Other proofs} 
\label{sec:keyLemmas}

In this section, we prove Theorem \ref{thm:gradient}. 
We begin with the formula for the gradient in Lemma \ref{lem:derivative}. Throughout this section, we have some edge $e = \{x,y\}$ in the tree $T$. We let $Z_x$ (resp.~$Z_y$) be the magnetization on the descendant subtree $T_x$ (resp.~$T_y$) with respect to $y$ (resp.~$x$).


\begin{proof}[\textbf{Proof of Lemma \ref{lem:derivative}}]
		Let $T$ be a binary tree and $\hparam = (\htheta_e;e\in E(T))$ denote a collection of edge parameters. 
  The law of a random state vector $\sigma = (\sigma_u; u\in T)$ on \textit{all} vertices of $T$ satisfies that for each $\tau\in \{\pm 1\}^{V(T)}$, 
		\begin{equation*}
			\PR_\hparam (\sigma = \tau) = \frac{1}{2} \prod_{e = \{u,v\} \in E(T)} \frac{1+\tau_u \tau_v \htheta_e}{2}.
		\end{equation*} 
		Marginalizing over the leaves
		and breaking things up according to 
		the spins at nodes $u$ and $v$, for a spin configuration $\tau\in \{\pm 1\}^{V(T)}$, 
		\begin{align*}
			\PR_\hparam(X_L = \tau_L) &= \PR(X_{L_u} = \tau_{L_u}, \sigma_{L_v} = \tau_{L_v})\\
			&=   \sum_{\tau'_u,\tau'_v\in\{-,+\} }\PR_\hparam(\sigma_{L_u} = \tau_{L_u}, \sigma_{L_v} = \tau_{L_v}| \sigma_u = \tau'_u, \sigma_v=\tau_v' )\, \PR_\hparam(\sigma_u = \tau'_u,\sigma_v = \tau'_v)\\
			&=\frac{1}{2} \sum_{\tau'_u,\tau'_v\in\{-,+\} }\PR_\hparam(\sigma_{L_u} = \tau_{L_u}, \sigma_{L_v} = \tau_{L_v}| \sigma_u = \tau'_u, \sigma_v=\tau_v' )\frac{1+\tau_u'\tau_v'\htheta_e}{2}.
		\end{align*} 
		By the Markov property and using 
		Bayes' rule to bring out the magnetization, 
		\begin{align*}
			\PR_\hparam(\sigma_L = \tau_L) 
			&=\frac{1}{2} \sum_{\tau'_u,\tau'_v\in\{-,+\} }\PR_\hparam(\sigma_{L_u} = \tau_{L_u}|\sigma_u = \tau'_u) \, \PR_{\hparam}( \sigma_{L_v} = \tau_{L_v}|  \sigma_v=\tau_v' )\frac{1+\tau_u'\tau_v'\htheta_e}{2}\\
			&= \frac{1}{2} \sum_{\tau'_u,\tau'_v\in\{-,+\} } \bigg\{
			\frac{\PR_\hparam(\sigma_u = \tau_u'|\sigma_{L_u} = \tau_{L_u})\PR_\hparam(\sigma_{L_u} = \tau_{L_u})}{\PR_\hparam(\sigma_u = \tau_u')}\\
			&\quad\qquad \qquad \qquad \qquad \times \frac{\PR_{\hparam}(\sigma_v = \tau_v'| \sigma_{L_v} = \tau_{L_v} )\PR_{\hparam}(\sigma_{L_v} = \tau_{L_v})}{\PR_{\hparam}(\sigma_v = \tau_v')}\frac{1+\tau_u'\tau_v'\htheta_e}{2}\bigg\}.
\end{align*}
Since $\PR_{\hparam}(\sigma_x = \tau_x) = \frac12$ for all $x\in T$, we have
\begin{align*}
			\PR_{\hparam}&(\sigma_L = \tau_L)= \PR_\hparam(\sigma_{L_u} = \tau_{L_u})\, \PR_\hparam(\sigma_{L_v} = \tau_{L_v})\\
			&\qquad\qquad \times\sum_{\tau_u',\tau_v'\in\{-,+\}} \PR_\hparam(\sigma_u = \tau_u'|\sigma_{L_u} = \tau_{L_u})\PR_\hparam(\sigma_v = \tau_v' | \sigma_{L_v} = \tau_{L_v}) (1+\tau_u'\tau_v'\htheta_e)\\
			&= \PR_\hparam(\sigma_{L_u} = \tau_{L_u})\PR_\hparam(\sigma_{L_v} = \tau_{L_v})\Bigg(1 + \htheta_e\sum_{\tau_u',\tau_v'\in\{-,+\}} \PR_\hparam(\sigma_u = \tau_u'|\sigma_{L_u} = \tau_{L_u})\PR_\hparam(\sigma_v = \tau_v' | \sigma_{L_v} = \tau_{L_v}) \tau_u'\tau_v' \Bigg)\\
			&=  \PR_{\hparam}(\sigma_{L_u}=\tau_{L_u})\PR_\hparam(\sigma_{L_v} = \tau_{L_v}) (1+\htheta_e Z_u Z_v ).
		\end{align*} 
		In the last two lines we used the fact that $\PR_\hparam(\sigma_u = \tau_u'|\sigma_{L_u} = \tau_{L_u})\PR_\hparam(\sigma_v = \tau_v' | \sigma_{L_v} = \tau_{L_v})$ is a product probability measure on $\{(\tau_u', \tau_v'): \tau_u',\tau_v'\in\{-,+\}\}$ and, under this measure, $\E[\tau_u'] = Z_u$ and $\E[\tau_v'] = Z_v$.
		
		It follows that 
		\begin{equation*}
			\ell(\hparam; \tau_L)  = \log \PR_{\hparam}(\sigma_{L}=\tau_{L})=  \log (1+Z_uZ_v\htheta_e) + \log \PR_\hparam(\sigma_{L_u} = \tau_{L_u})+\log \PR_\hparam(\sigma_{L_v} = \tau_{L_v}).
		\end{equation*} 
		Observing that neither $\PR_\hparam(\sigma_{L_u} = \tau_{L_u})$
		nor
		$\PR_\hparam(\sigma_{L_v} = \tau_{L_v})$ depend on $\htheta_e$ and hence have vanishing derivatives,
		one finally obtains \eqref{eqn:derivative}.
\end{proof}



\begin{proof}[\textbf{Proof of Theorem} \ref{thm:gradient}] 
		Recall that by the computation in \eqref{eqn:grad00}, equation \eqref{eqn:gradient} is equivalent to
\begin{align}\label{eqn:grad_thm_claim1}
			\left| \E_{\param^{*}}\left[\frac{Z_x Z_y}{1+\htheta_e Z_xZ_y}\right]-\E_{\param^{*}}\left[\frac{\sigma_x \sigma_y}{1+\htheta_e\sigma_x\sigma_y}\right]\right|\le C_{\ref{eqn:gradient}}\delta.
		\end{align}		
  While the magnetizations $Z_{x}$ and $Z_{y}$ are likely to be close to the spins $\sigma_{x}$ and $\sigma_{y}$ respectively, there can be a significant discrepancy with small probability and we need to quantify how this discrepancy affects the first expectation in \eqref{eqn:grad_thm_claim1}. 
  
  To do so, in view of the reconstruction lemma (Theorem \ref{thm:Robust}), we divide up the magnetizations into three cases following \eqref{eq:recons_tiers}. 
		Define the events 
		\begin{align}\label{eq:events_grad_thm_pf}
			\begin{cases}
				\quad 	F&:= \{ \textup{flip on the edge $e=\{x,y\}$} \}= \{\sigma_x\neq \sigma_y\} \\
				\quad	A&:= \{ \textup{good reconstruction at both $x$ and $y$}\} \\
				\quad	M&:= \{ \textup{one moderate failure and one good reconstruction among $x$ and $y$} \}\\
				\quad	E&:= A \cup (F^c\cap M). 
			\end{cases}
		\end{align}
		Note that the events $A$, $F^c\cap M$, and $E^{c}$ partition the sample space. 
		In order to show \eqref{eqn:grad_thm_claim1}, we will compute the expectations on the ``good'' event $A$ and the ``bad'' event $A^{c}=E^{c}\sqcup (F^{c}\cap M)$.

        \medskip\noindent\textbf{Good reconstructions} We start with the event $A$ where we can successfully recover the spins $\sigma_x,\sigma_y$. We will show 
		\begin{align}\label{eqn:grad_thm_pf2}
			\left| \frac{Z_xZ_y}{1+\htheta_eZ_xZ_y}- \frac{\sigma_x\sigma_y}{1+\sigma_x\sigma_y \htheta_e} \right|\mathbf{1}_A \le O(\delta^2) \mathbf{1}_{F^c} + O(1) \mathbf{1}_F. 
		\end{align} 
		Since $\PR_{\param^{*}}(F)\le C_{\ref{eqn:pBounds}}\delta$ by \ref{assumption1}, this will yield that the expectation under $\E_{\param^{*}}$ of the left-hand side above is $O(\delta)$. 
		
		To justify \eqref{eqn:grad_thm_pf2} we break up $A$ into two cases according to the flip event $F$. Namely, on $F^c$ we have $\sigma_x\sigma_y = 1$ and on $F$ we have $\sigma_x\sigma_y = -1$. Hence, the left-hand side of \eqref{eqn:grad_thm_pf2} is
  \begin{align}
      \nonumber &\left| \frac{Z_xZ_y}{1+\htheta_eZ_xZ_y}- \frac{\sigma_x\sigma_y}{1+\sigma_x\sigma_y \htheta_e} \right|\mathbf{1}_A = \left| \frac{Z_xZ_y}{1+\htheta_eZ_xZ_y}- \frac{1}{1+ \htheta_e} \right|\mathbf{1}_{A\cap F^c} + \left| \frac{Z_xZ_y}{1+\htheta_eZ_xZ_y}- \frac{-1}{1-\htheta_e} \right|\mathbf{1}_{A\cap F}\\
      &\qquad\qquad = \left|\frac{1-Z_xZ_y}{(1+\htheta_e)(1+\htheta_e Z_xZ_y)}\right| \mathbf{1}_{A\cap F^c} + \left|\frac{1+Z_xZ_y}{(1-\htheta_e)(1-\htheta_e+\htheta_e(1+Z_xZ_y))}\right|\mathbf{1}_{A\cap F}\label{eqn:grad_helper_10.9_1}
  \end{align}

  The key insight is that both numerators in \eqref{eqn:grad_helper_10.9_1} can be controlled by the approximation $Z_{x}Z_{y}\approx \sigma_{x}\sigma_{y}$ on $A$. More quantitatively, noting that $|\sigma_{x}\sigma_{y}|=1$, 
		\begin{align}\label{eqn:grad5}
			\left|\sigma_x \sigma_y - Z_xZ_y \right|\mathbf{1}_A  =	\left|\sigma_x \sigma_y - (\sigma_{x}\sigma_{y})^{2}Z_xZ_y \right|\mathbf{1}_A  = \left| 1 - (\sigma_xZ_x)(\sigma_xZ_y)  \right| \mathbf{1}_{A}= O(\delta^2), 
		\end{align}
		where the last inequality follows from the definition of the event $A$. 
		
  For the $F^c$ term in \eqref{eqn:grad_helper_10.9_1}, we lower bound the denominator using $\hat{\theta}_{e}\ge 0$  (see \ref{assumption1}) and  $Z_{x}Z_{y}\ge 1- 8 C_{\ref{eqn:pHatBounds}}\delta^2\ge 0$ on the event $A$.  Hence
  \begin{align*}
			\left| \frac{1-Z_{x}Z_{y}}{(1+ \htheta_e)\left( 1+ \htheta_e  Z_{x}Z_{y}\right)}\right|\mathbf{1}_{A\cap F^{c}} =O(\delta^2) \mathbf{1}_{F^{c}}.
		\end{align*}
  
  Similarly, on $A\cap F$, we have $1+Z_xZ_y\ge 0$ and $1-\htheta_e \ge 2c_{\ref{eqn:pHatBounds}}\delta$ from \ref{assumption1} which allow us to control the denominator in the corresponding term in \eqref{eqn:grad_helper_10.9_1}. Namely,
		\begin{align*}
			 \left| \frac{1+Z_{x}Z_{y}}{(1-\htheta_e)\left( 1 -\htheta_e  +\htheta_{e}(1+Z_{x}Z_{y})\right)}\right|\mathbf{1}_{A\cap F} \le \frac{|1+Z_xZ_y|}{(1-\htheta_e)^2} \mathbf{1}_{A\cap F} \le \frac{O(\delta^2)}{(2c_{\ref{eqn:pHatBounds}}\delta)^{2}} \mathbf{1}_{F} = O(1) \mathbf{1}_{F} .
		\end{align*} The last two displayed equations imply \eqref{eqn:grad_thm_pf2}.
  
		\medskip\noindent\textbf{Moderate failure, but no flip} 
		We move to the event $F^c\cap M$, where $\sigma_x = \sigma_y$ (the event $F^c$) but there is a moderate failure at one of the vertices (the event $M$). On this event, we have 
		\begin{align*}
			1+\htheta_e Z_xZ_y \ge 1 + (1-C_{\ref{eqn:Reconstruct}} \delta^2)(-c_{\ref{eqn:antiReconstruction}})  \ge 1-c_{\ref{eqn:antiReconstruction}}
		\end{align*}
		and also 
		\begin{align*}
			1+\htheta_e \sigma_x \sigma_y  = 1+\hat{\theta}_{e} \ge 2-2C_{\ref{eqn:pBounds}}\delta.
		\end{align*} 
		By Theorem \ref{thm:Robust} and a union bound, $\PR_{\param^{*}}(M)\le 2c_{\ref{eqn:Reconstruct}} \delta = O(\delta)$. It follows that 
		\begin{align}\label{eqn:grad3}
		\left|\E_{\param^{*}}\left[\frac{Z_xZ_y}{1+\htheta_eZ_xZ_y} \mathbf{1}_{F^c\cap M}\right]\right| ,\, 	\left|\E_{\param^{*}}\left[\frac{\sigma_x \sigma_y}{1+\htheta_e \sigma_x \sigma_y} \mathbf{1}_{F^c\cap M}\right]\right| & \le \frac{1}{(1-c_{\ref{eqn:antiReconstruction}})\land (2-2C_{\ref{eqn:pBounds}})} \PR(M) 
 = O(\delta).
		\end{align}
		
		\medskip\noindent\textbf{Severe failure or moderate failure and flip} Lastly, consider the event $E^{c} = A^c \cap (F\cup M^c)$. Recall that by Claim \ref{lem:signed_mag}, the unsigned magnetizations $\sigma_{x}Z_{x}$ and $\sigma_{y}Z_{y}$, and $\mathbf{1}(\sigma_{x}\ne \sigma_{y})$ are all independent. Note that the event $A^{c}\cap F$ means at least one of $x$ and $y$ does not have good reconstruction and there is flip on $e$, which occurs with probability $\le (2C_{\ref{eqn:Reconstruct}} \delta)p_{e}\le (2C_{\ref{eqn:Reconstruct}} \delta)C_{\ref{eqn:pBounds}}\delta$ (by Theorem \ref{thm:Robust}, independence, and \ref{assumption1}). Furthermore, the event $A^{c}\cap M^{c}$ is that there are at least one severe failure or two moderate failures (recall the terminologies in  \eqref{eq:recons_tiers}) among $x$ and $y$. The former occurs with probability $\le 2C_{\ref{eqn:antiReconstruction}}\delta^{2}$ (by a union bound and Theorem \ref{thm:Robust}) and the latter occurs with probability $\le c_{\ref{eqn:Reconstruct}}^{2}\delta^{2}$ (by Theorem \ref{thm:Robust} and the independence of the unsigned magnetizations). It follows that 
		\begin{align} 
		\nonumber	\PR_{\param^{*}}( E^{c}) &= \PR_{\param^{*}}(A^{c}\cap (F\cup M^{c}) )\\
  &\nonumber \le  \PR_{\param^{*}}(A^{c}\cap F)+\PR_{\param^{*}}(A\cap  M^{c}) \\
  &\le 2C_{\ref{eqn:Reconstruct}} C_{\ref{eqn:pBounds}}\delta^2 +   {(2C_{\ref{eqn:antiReconstruction}}+c_{\ref{eqn:Reconstruct}}^{2})}\delta^{2} = O(\delta^2).\label{eqn:grad2}
		\end{align} 
		Recall the generic bound
		\begin{equation}\label{eqn:grad1}
			\left| \frac{Z_xZ_y}{1+\htheta_e Z_xZ_y} \right|, \left|\frac{\sigma_x\sigma_y}{1+\sigma_x\sigma_y \htheta_e} \right| \le \frac{1}{2c_{\ref{eqn:pHatBounds}}\delta},
		\end{equation}
		which follows since $|Z_xZ_y|\in[-1,1]$ and $\sigma_x\sigma_y \in \{-1,1\}$ and $\htheta_e\le 1-2c_{\ref{eqn:pHatBounds}}\delta$ under Assumption \ref{assumption1}. Hence combining \eqref{eqn:grad2} and \eqref{eqn:grad1}, we get 
		\begin{align}\label{eqn:grad11}
			\left| 	\E_{\param^{*}}\left[\frac{Z_xZ_y}{1+\htheta_e Z_xZ_y}\mathbf{1}_{E^{c}}\right]\right|,\, \left| 	\E_{\param^{*}}\left[\frac{\sigma_x \sigma_y}{1+\htheta_e \sigma_x \sigma_y}\mathbf{1}_{E^{c}}\right]\right| \le O(\delta^{-1})\PR_{\param^{*}}(E^c) = O(\delta^{-1})\times O(\delta^2) = O(\delta).
		\end{align}	
		Now since the events $A$, $F^c\cap M$, and $E^{c}$ partition the sample space, using \eqref{eqn:grad3}, \eqref{eqn:grad1}, \eqref{eqn:grad11}, with the triangle inequality, we deduce 
        \begin{align*}
            &\left| \E_{\param^{*}}\left[\frac{Z_x Z_y}{1+\htheta_e Z_xZ_y}\right]-\E_{\param^{*}}\left[\frac{\sigma_x \sigma_y}{1+\htheta_e\sigma_x\sigma_y}\right]\right|\\
            &\le\left| \E_{\param^{*}}\left[\frac{Z_x Z_y}{1+\htheta_e Z_xZ_y}\mathbf{1}_A\right]-\E_{\param^{*}}\left[\frac{\sigma_x \sigma_y}{1+\htheta_e\sigma_x\sigma_y}\mathbf{1}_A\right]\right| + \left| \E_{\param^{*}}\left[\frac{Z_x Z_y}{1+\htheta_e Z_xZ_y}\mathbf{1}_{A^c}\right]\right|+\left|\E_{\param^{*}}\left[\frac{\sigma_x \sigma_y}{1+\htheta_e\sigma_x\sigma_y}\mathbf{1}_{A^c}\right]\right|\\
            &= O(\delta)+ O(\delta)+O(\delta) = O(\delta).
        \end{align*}
        This is precisely \eqref{eqn:grad_thm_claim1}.
	\end{proof}

	Lastly in this section, we derive Corollary \ref{cor:coord_ini} from Theorem \ref{thm:gradient}.
	
	\begin{proof}[\textbf{Proof of Corollary \ref{cor:coord_ini}}]
		
		By Theorem \ref{thm:gradient}, there is some small signed error $\eps\in[-C_{\ref{eqn:gradient}}\delta,C_{\ref{eqn:gradient}}\delta]$ such that 
		\begin{equation*}
			\frac{\partial}{\partial\htheta_e} \E_\param\left[\ell(\hparam;\sigma)\right] = \frac{\theta_e-\htheta_e}{(1-\htheta_e^2)} + \eps.
		\end{equation*} 
        Hence solving $\frac{\partial}{\partial\htheta_e} \E_\param\left[\ell(\hparam;\sigma)\right] =0$ amounts to solving the following equation in $\htheta_e$
		\begin{equation*}
			\eps\htheta_e^2 +\htheta_e - (\theta_e + \eps)=0.
		\end{equation*}
		The relevant root is
		\begin{align*}
			\htheta_e^* = \frac{-1+\sqrt{1+4\eps(\theta_e+\eps)}}{2\eps} = \frac{1}{2\eps}\left(-1+ 1+\frac{4\eps(\theta_e+\eps)}{2}- \frac{16\eps^2(\theta_e+\eps)^2}{8} +\eta\right),
		\end{align*} where $|\eta|\le \frac{1}{16} |4\eps(\theta_e+\eps)|^3 \le 32\eps^3$. The bound on $\eta$ follows from a Taylor expansion and the fact that $\theta+\eps\le 2$. 
  Recalling that $\theta_e = 1-2p_e$, 
  we have 
		\begin{align*}
			\htheta_e^* &= \frac{1}{2\eps} \left( 2\eps(\theta_e + \eps) - 2\eps^2(1-2p_e+\eps)^2\right)+\frac{\eta}{\eps}\\
			&= \theta_e  + 4p_e\eps + 4p_e\eps^2- 2\eps^2 +\eps^3 + 4\eps p_e^2 + \frac{\eta}{\eps}.
		\end{align*}
		Therefore, 
        using the bounds $p_{e}\le C_{\ref{eqn:pBounds}}\delta$, $|\eps|\le C_{\ref{eqn:gradient}}\delta$, and $|\eta| \le 32\eps^{3}$, we deduce
		\begin{align*}
			|\htheta_e^{*} - \theta_e| &\le \left|  4p_e\eps + (4p_e+2)\eps^2 +\eps^3 + 4\eps p_e^2 + \frac{\eta}{\eps} \right| \le C_{\ref{eqn:coord_ini_guarantee}} \delta^{2}
		\end{align*} 
		for some constant $C_{\ref{eqn:coord_ini_guarantee}}>0$. 
	\end{proof}

    \section*{Acknowledgements}

    HL was partially supported by NSF grant DMS-2206296. DC and SR were partially supported by the Institute for Foundations of Data Science (IFDS) through NSF grant DMS-2023239 (TRIPODS Phase II). It is also based upon work supported by the NSF under grant DMS-1929284 while one of the authors (SR) was in residence at the Institute for Computational and Experimental Research in Mathematics (ICERM) in Providence, RI, during the Theory, Methods, and Applications of Quantitative Phylogenomics semester program. SR was also supported by NSF grant DMS-2308495, as well as a Van Vleck Research Professor Award and a Vilas Distinguished Achievement Professorship.


\begin{thebibliography}{10}

\bibitem{bhatnagar2011reconstruction}
N.~Bhatnagar, J.~Vera, E.~Vigoda, and D.~Weitz.
\newblock Reconstruction for colorings on trees.
\newblock {\em SIAM Journal on Discrete Mathematics}, 25(2):809--826, 2011.

\bibitem{BlRuZa:95}
P.~M. Bleher, J.~Ruiz, and V.~A. Zagrebnov.
\newblock On the purity of the limiting {G}ibbs state for the {I}sing model on the {B}ethe lattice.
\newblock {\em J. Statist. Phys.}, 79(1-2):473--482, 1995.

\bibitem{BCMR.06}
C.~Borgs, J.~Chayes, E.~Mossel, and S.~Roch.
\newblock The {K}esten-{S}tigum reconstruction bound is tight for roughly symmetric binary channels.
\newblock In {\em 2006 47th Annual IEEE Symposium on Foundations of Computer Science (FOCS'06)}, pages 518--530. IEEE, 2006.

\bibitem{brent2013algorithms}
R.~P. Brent.
\newblock {\em Algorithms for minimization without derivatives}.
\newblock Courier Corporation, 2013.

\bibitem{cavender1978taxonomy}
J.~A. Cavender.
\newblock Taxonomy with confidence.
\newblock {\em Mathematical biosciences}, 40(3-4):271--280, 1978.

\bibitem{DMR.11}
C.~Daskalakis, E.~Mossel, and S.~Roch.
\newblock Evolutionary trees and the {I}sing model on the {B}ethe lattice: a proof of {S}teel's conjecture.
\newblock {\em Probab. Theory Related Fields}, 149(1-2):149--189, 2011.

\bibitem{dinh_matsen_2017}
V.~Dinh and F.~A.~M. IV.
\newblock {The shape of the one-dimensional phylogenetic likelihood function}.
\newblock {\em The Annals of Applied Probability}, 27(3):1646 -- 1677, 2017.

\bibitem{EvKePeSc:00}
W.~S. Evans, C.~Kenyon, Y.~Peres, and L.~J. Schulman.
\newblock Broadcasting on trees and the {I}sing model.
\newblock {\em Ann. Appl. Probab.}, 10(2):410--433, 2000.

\bibitem{FanRoch:18}
W.-T.~L. Fan and S.~Roch.
\newblock Necessary and sufficient conditions for consistent root reconstruction in {M}arkov models on trees.
\newblock {\em Electron. J. Probab.}, 23:24 pp., 2018.

\bibitem{farris1973probability}
J.~S. Farris.
\newblock A probability model for inferring evolutionary trees.
\newblock {\em Systematic Biology}, 22(3):250--256, 1973.

\bibitem{fukami_maximum_1989}
K.~Fukami and Y.~Tateno.
\newblock On the maximum likelihood method for estimating molecular trees: {Uniqueness} of the likelihood point.
\newblock {\em Journal of Molecular Evolution}, 28(5):460--464, May 1989.

\bibitem{guindon2003simple}
S.~Guindon and O.~Gascuel.
\newblock A simple, fast, and accurate algorithm to estimate large phylogenies by maximum likelihood.
\newblock {\em Systematic biology}, 52(5):696--704, 2003.

\bibitem{Ioffe:96a}
D.~Ioffe.
\newblock On the extremality of the disordered state for the {I}sing model on the {B}ethe lattice.
\newblock {\em Lett. Math. Phys.}, 37(2):137--143, 1996.

\bibitem{JansonMossel:04}
S.~Janson and E.~Mossel.
\newblock Robust reconstruction on trees is determined by the second eigenvalue.
\newblock {\em Ann. Probab.}, 32:2630--2649, 2004.

\bibitem{kesten1966additional}
H.~Kesten and B.~P. Stigum.
\newblock Additional limit theorems for indecomposable multidimensional galton-watson processes.
\newblock {\em The Annals of Mathematical Statistics}, 37(6):1463--1481, 1966.

\bibitem{makur_broadcasting_2020}
A.~Makur, E.~Mossel, and Y.~Polyanskiy.
\newblock Broadcasting on {Random} {Directed} {Acyclic} {Graphs}.
\newblock {\em IEEE Transactions on Information Theory}, 66(2):780--812, Feb. 2020.
\newblock Conference Name: IEEE Transactions on Information Theory.

\bibitem{Mossel:01}
E.~Mossel.
\newblock Reconstruction on trees: beating the second eigenvalue.
\newblock {\em Ann. Appl. Probab.}, 11(1):285--300, 2001.

\bibitem{Mossel:04a}
E.~Mossel.
\newblock Phase transitions in phylogeny.
\newblock {\em Trans. Amer. Math. Soc.}, 356(6):2379--2404, 2004.

\bibitem{mossel:22}
E.~Mossel.
\newblock Combinatorial statistics and the sciences.
\newblock In {\em Proceedings of the ICM}, 2022.

\bibitem{MosselPeres:03}
E.~Mossel and Y.~Peres.
\newblock Information flow on trees.
\newblock {\em Ann. Appl. Probab.}, 13(3):817--844, 2003.

\bibitem{neyman1971molecular}
J.~Neyman.
\newblock Molecular studies of evolution: a source of novel statistical problems.
\newblock In {\em Statistical decision theory and related topics}, pages 1--27. Elsevier, 1971.

\bibitem{Roch:10}
S.~Roch.
\newblock {Toward Extracting All Phylogenetic Information from Matrices of Evolutionary Distances}.
\newblock {\em Science}, 327(5971):1376--1379, 2010.

\bibitem{Roch_2024}
S.~Roch.
\newblock {\em Modern Discrete Probability: An Essential Toolkit}.
\newblock Cambridge Series in Statistical and Probabilistic Mathematics. Cambridge University Press, 2024.

\bibitem{RochSly:17}
S.~Roch and A.~Sly.
\newblock Phase transition in the sample complexity of likelihood-based phylogeny inference.
\newblock {\em Probability Theory and Related Fields}, 169(1):3--62, Oct 2017.

\bibitem{Sly:11}
A.~Sly.
\newblock Reconstruction for the {P}otts model.
\newblock {\em Ann. Probab.}, 39(4):1365--1406, 07 2011.

\end{thebibliography}
\end{document}